\documentclass[11pt]{amsart}

\usepackage{amsmath}
\usepackage{amssymb}
\usepackage{amsthm}
\usepackage[mathscr]{eucal}
\usepackage{mathrsfs}
\usepackage{psfrag}
\usepackage{fullpage}
\usepackage{color}
\usepackage{eucal}
\usepackage[dvips]{graphicx}
\usepackage{amsfonts}%
\usepackage{amsaddr}
\providecommand{\U}[1]{\protect\rule{.1in}{.1in}}

\newtheorem{proposition}{Proposition}[section]
\newtheorem{theorem}[proposition]{Theorem}

\newtheorem{lemma}[proposition]{Lemma}

\newtheorem{definition}[proposition]{Definition}
\newtheorem{remark}[proposition]{Remark}
\newtheorem{example}[proposition]{Example}

\newtheorem{condition}[proposition]{Condition}

\numberwithin{equation}{section}
\numberwithin{proposition}{section}

\begin{document}
\title{Rare event simulation for multiscale diffusions in random environments}

\author{Konstantinos Spiliopoulos}
\address{Department of Mathematics \& Statistics\\
Boston University\\
Boston, MA 02215}
\email{kspiliop@math.bu.edu}

\date{\today}

\begin{abstract}
We consider systems of stochastic differential equations with multiple scales and small noise and assume that the coefficients of the equations are  ergodic and stationary random fields. Our goal is to construct provably-efficient importance sampling Monte Carlo methods that allow efficient computation of rare event probabilities or expectations of functionals that can be associated with rare events. Standard Monte Carlo algorithms perform poorly in the small noise limit and hence fast simulations algorithms become relevant. The presence of multiple scales complicates the design and the analysis of efficient importance sampling schemes. An additional complication is the randomness of the environment. We construct explicit changes of measures that are proven to be logarithmic asymptotically efficient with probability one with respect to the random environment (i.e., in the quenched sense). Numerical simulations support the theoretical results.
\end{abstract}



\maketitle

\section{Introduction}

\label{S:Intro}

In this paper, we study the construction of logarithmic asymptotically optimal importance sampling methods for small noise multiscale diffusion processes in random environments. Importance sampling is a variance reduction technique in Monte-Carlo simulation and it is in particularly useful when one deals with rare events.   For example, consider a sequence $\{X^{\epsilon}\}_{\epsilon>0}$ of random elements and assume that we want to estimate the probability $0<\mathbb{P}\left[X^{\epsilon}\in A\right]\ll 1$ for a given set $A$, such that the event $\left\{X^{\epsilon}\in A\right\}$ is unlikely for small $\epsilon$. If one is interested in accurate estimation of quantities such as $p^{\epsilon}=\mathbb{P}\left[X^{\epsilon}\in A\right]$  for $\epsilon$ small, and closed form formulas are not available, or numerical approximations are either too crude or unavailable, then one has to resort in simulation. However, as it is well known, standard Monte-Carlo simulation techniques (i.e., using the unbiased estimator $\hat{p}^{\epsilon}=\frac{1}{N}\sum_{j=1}^{N}1_{X^{\epsilon,j}\in A}$) perform rather poorly in the rare-event regime. In particular, for standard Monte-Carlo it is known that in order to achieve relative error smaller than one, one needs an effective sample size $N\approx 1/p^{\epsilon}$, e.g. see \cite{AsmussenGlynn2007}. Namely, for a fixed computational
cost, relative errors grow rapidly as the event becomes more rare.   This makes standard Monte-Carlo infeasible for rare-event simulation. Thus, one is lead to construct accelerated Monte-Carlo methods that reduce the variance of the estimators, and a popular method to do so is importance sampling.

In importance sampling one identifies an appropriate change of measure and simulates the system under the new measure, which is chosen in order to minimize the variance of the estimator. This is exactly the problem that this paper addresses for small-noise, multiscale diffusion processes in random environments. We consider the following general class of such models. Let $0<\varepsilon,\delta\ll 1$ and consider the process $\left(X^{\epsilon}, Y^{\epsilon}\right)=\left\{\left(X^{\epsilon}_{t}, Y^{\epsilon}_{t}\right), t\in[0,T]\right\}$ taking values in the space $\mathbb{R}^{m}\times\mathbb{R}^{d-m}$ that satisfies the system of stochastic differential equation (SDE's)

\begin{eqnarray}
dX^{\epsilon}_{t}&=&\left[  \frac{\epsilon}{\delta}b\left(Y^{\epsilon}_{t},\gamma\right)+c\left(  X^{\epsilon}_{t}%
,Y^{\epsilon}_{t},\gamma\right)\right]   dt+\sqrt{\epsilon}%
\sigma\left(  X^{\epsilon}_{t},Y^{\epsilon}_{t},\gamma\right)
dW_{t},\nonumber\\
dY^{\epsilon}_{t}&=&\frac{1}{\delta}\left[  \frac{\epsilon}{\delta}f\left(Y^{\epsilon}_{t},\gamma\right)  +g\left(  X^{\epsilon}_{t}%
,Y^{\epsilon}_{t},\gamma\right)\right] dt+\frac{\sqrt{\epsilon}}{\delta}\left[
\tau_{1}\left(  Y^{\epsilon}_{t},\gamma\right)
dW_{t}+\tau_{2}\left(Y^{\epsilon}_{t},\gamma\right)dB_{t}\right], \label{Eq:Main}\\
X^{\epsilon}_{0}&=&x_{0},\hspace{0.2cm}Y^{\epsilon}_{0}=y_{0}\nonumber
\end{eqnarray}

We assume that $\delta=\delta(\epsilon)\downarrow0$ such that $\epsilon/\delta\uparrow\infty$ as $\epsilon\downarrow0$.  $(W_{t}, B_{t})$ is a $2\kappa-$dimensional standard Wiener process. We assume that for each fixed $x\in\mathbb{R}^{m}$,  $b(\cdot,\gamma), c(x,\cdot,\gamma),\sigma(x,\cdot,\gamma),f(\cdot,\gamma)$, $g(x,\cdot,\gamma), \tau_{1}(\cdot,\gamma)$ and $\tau_{2}(\cdot,\gamma)$ are stationary and ergodic random fields in an appropriate probability space $\left(\Gamma,\mathcal{G},\nu\right)$ with $\gamma\in\Gamma$.

One can interpret the system (\ref{Eq:Main}) as a small noise perturbation of a system of slow and fast components with $X^{\epsilon}$ being the slow and $Y^{\epsilon}$ being the fast component. As it is seen from (\ref{Eq:MainEx1}) and (\ref{Eq:MainEx2}) below, classical cases of interest are special cases of (\ref{Eq:Main}). The primary goal of this article is to provide a rigorous mathematical framework for the development of provably-efficient  and asymptotically optimal importance sampling Monte-Carlo methods. Given a realization of the random medium $\gamma$, our goal is to estimate quantities of the form
\begin{equation}
\theta(\epsilon,\gamma)=\mathbb{E}_{x_{0},y_{0}}\left[e^{-\frac{1}{\epsilon}h(X^{\epsilon,\gamma}_{T})} \right]\quad\text{ or }\quad \theta(\epsilon,\gamma)=\mathbb{P}_{x_{0},y_{0}}\left[X^{\epsilon,\gamma}_{T}\in A\right]\label{Eq:QuantitiesOfInterest}
\end{equation}
such that the estimators are known to be provably-efficient a-priori for almost all realizations of the random environment (quenched asymptotics). 
If one is interested in accurate estimation of quantities as in   (\ref{Eq:QuantitiesOfInterest}), then one does not have any hope of closed form formulas and logarithmic large deviation estimates are too crude (since they ignore potential important prefactors) and thus simulation becomes necessary.

Our work is partially motivated by related questions in chemical physics, molecular dynamics, climate modeling and neuroscience, e.g., \cite{EVMaidaTimofeyev2001,DupuisSpiliopoulosWang,DupuisSpiliopoulosWang2,Jirsa_etall,SchutteWalterHartmannHuisinga2005, Zwanzig}.  There, one is often interested in simplified models that preserve the large deviation properties of the system in the case where $\delta\ll\epsilon$, i.e., in the case where $\delta$ is orders of magnitude smaller than $\epsilon$.  In the large deviations regime, rare events dominate and then estimation of probabilities of rare events or of related expectations of functionals of interest becomes a challenging issue, since standard Monte Carlo perform poorly. In this paper, we provide a provably asymptotically-efficient importance sampling scheme for estimation of quantities such as (\ref{Eq:QuantitiesOfInterest}). The change of measure is motivated by the quenched large deviations results of \cite{Spiliopoulos2014}. In \cite{Spiliopoulos2014}, we proved the quenched, i.e., almost surely with respect to the random environment,  large deviations principle for (\ref{Eq:Main}). In particular, the control achieving the large deviations bound gives useful information, which motivates the design of provably efficient importance sampling schemes for estimation of related rare event probabilities.

To the best of the author's knowledge, this is the first work that addresses design of logarithmic asymptotically efficient importance sampling methods for multiscale diffusions in random environments. In the absence of fast oscillations, importance sampling schemes for small noise diffusions have been recently considered in \cite{DupuisSpiliopoulosZhou2013, DupuisSpiliopoulos2014, Spiliopoulos2014a, VandenEijndenWeare}. The authors in \cite{DupuisSpiliopoulosZhou2013, DupuisSpiliopoulos2014, Spiliopoulos2014a} investigate the impact of metastability effects in the design and analysis of non-asymptotic provably-efficient importance sampling schemes for small noise diffusions, but without multiple scales.  It was found there that in the presence of metastability, naive changes of measure could perform poorly even if they are asymptotically optimal. Hence, in these problems, non-asymptotic optimality is studied and importance sampling schemes are constructed that have guaranteed good performance even non-asymptotically. However, these papers do not study the effect of multiple scales, see Remark \ref{R:MetastabilityEffect}.  In the case of a periodic fast motion, the design of large deviations inspired efficient Monte Carlo importance sampling schemes was investigated in \cite{DupuisSpiliopoulosWang,DupuisSpiliopoulosWang2,Spiliopoulos2013}.  The present work is also related to the theory of random homogenization of HJB equations
\cite{KomorowskiLandimOlla2012, KosyginaRezakhanlouVaradhan, Kozlov1979, LionsSouganidis2006,
Olla1994,OllaSiri2004,Osada1983,Osada1987, PapanicolaouVaradhan1982,Papaanicolaou1994, Rhodes2009a}. Other related models where the regime of interest is $\epsilon/\delta\uparrow \infty$
have been considered in \cite{Baldi, DupuisSpiliopoulos, DupuisSpiliopoulosWang,  FS, HorieIshii, Spiliopoulos2013}.
In the current paper, we consider the case of a random environment which complicates the analysis. Nevertheless, in the end we are able to construct logarithmic asymptotically efficient importance sampling changes of measure.

Using the quenched large deviation and weak convergence results from \cite{Spiliopoulos2014}, we construct
asymptotically optimal importance sampling schemes with rigorous bounds on
performance. The construction is based on two ingredients.  The first ingredient is the gradient of subsolutions to the associated homogenized Hamilton-Jacobi-Bellman (HJB) equation as in \cite{DupuisWang2}.
But, this constructions is relevant only after necessary modifications that take into account the multiscale aspect of the problem.
As it is also the case for the periodic medium, see \cite{DupuisSpiliopoulosWang},   changes of measure that
are purely based on the homogenized system and directly suggested by its
associated partial differential equation do not lead to efficient importance
sampling schemes. The second ingredient is used only when $b\neq0$ (see Remark \ref{R:NoUnboundedDriftTerm}) and is the macroscopic problem, or otherwise corrector, from the random homogenization theory. In the case $b\neq 0$, the corrector needs to be taken into account in order to achieve asymptotic optimality.

This construction  is motivated by the quenched large deviations results of  \cite{Spiliopoulos2014}, where a
change of measure (or equivalently a control) in partial feedback form has to
be used to prove a large deviation lower bound.  In this  paper we define a control which is of  feedback form, i.e., a function of both the slow variable $X^{\epsilon}$
and the fast variable $Y^{\epsilon}$, and which is used to construct dynamic
importance sampling schemes with precise asymptotic performance bounds.

Before closing this section we remark that model (\ref{Eq:Main}) contains as special cases many classical cases of interest. For example, if $b=g=0$ and we denote $\frac{\epsilon}{\delta^{2}}=\frac{1}{\eta}\gg 1$, then we get the system
\begin{eqnarray}
dX_{t}&=& c\left(  X_{t}%
,Y_{t},\gamma\right)   dt+\sqrt{\epsilon}%
\sigma\left(  X_{t},Y_{t},\gamma\right)
dW_{t},\nonumber\\
dY_{t}&=&\frac{1}{\eta}f\left(Y_{t},\gamma\right) dt+\frac{1}{\sqrt{\eta}}\left[
\tau_{1}\left(  Y_{t},\gamma\right)
dW_{t}+\tau_{2}\left(Y_{t},\gamma\right)dB_{t}\right], \label{Eq:MainEx1}\\
X_{0}&=&x_{0},\hspace{0.2cm}Y_{0}=y_{0}\nonumber
\end{eqnarray}
which can be interpreted as a small noise perturbation of the random dynamical system $\dot{X}=c(X,Y,\gamma)$, where $Y$ is stochastic and fast oscillating.  Moreover, if one considers the classical system, see for example \cite{BLP,PS}, of slow-fast variables
\begin{eqnarray}
dX_{t}&=&\left[  \frac{1}{\delta}b\left(Y_{t},\gamma\right)+c\left(  X_{t}%
,Y_{t},\gamma\right)\right]   dt+%
\sigma\left(  X_{t},Y_{t},\gamma\right)
dW_{t},\nonumber\\
dY_{t}&=&\frac{1}{\delta^{2}}f\left(Y_{t},\gamma\right) dt+\frac{1}{\delta}\left[
\tau_{1}\left(  Y_{t},\gamma\right)
dW_{t}+\tau_{2}\left(Y_{t},\gamma\right)dB_{t}\right], \label{Eq:MainEx2}\\
X_{0}&=&x_{0},\hspace{0.2cm}Y_{0}=y_{0}\nonumber
\end{eqnarray}
and rescale time $t\mapsto\epsilon t$, say if interested in small time asymptotics, then the process $(X^{\epsilon}_{s},Y^{\epsilon}_{s})=(X_{\epsilon s},Y_{\epsilon s})$ satisfies  (\ref{Eq:Main}) with $g=0$ and $c(x,y,\gamma)$ replaced by the lower order term $\epsilon c(x,y,\gamma)$ (which also vanishes in the limit).

In relation to (\ref{Eq:Main}), a special case of interest is the choice $b(y,\gamma)=f(y,\gamma)=-\nabla_{y}Q(y,\gamma)$  for an ergodic and stationary random field $Q(\cdot)$, $c(x,y,\gamma)=-\nabla_{x}V(x)$,  $\sigma(x,y,\gamma)=\tau_{1}(y,\gamma)=\sqrt{2D}$ and $\tau_{2}(y,\gamma)=0$. In this case, we simply have $Y^{\epsilon}_{t}=X^{\epsilon}_{t}/\delta$ and the model can be interpreted as  diffusion in the rough potential $\epsilon Q(x/\delta,\gamma)+ V(x)$, where the roughness is dictated by the random field $Q$, see \cite{Zwanzig} and \cite{DupuisSpiliopoulosWang,DupuisSpiliopoulosWang2} for further numerical studies. The random field $Q(\cdot)$ can be taken to be Gaussian with some specific correlation structure, see Section \ref{S:Examples}.

The rest of the paper is organized as follows. In Section \ref{S:Notation} we set-up notation, pose the main assumptions of the paper and recall the definition and well-known properties of the random environment that will be used in this paper. In Section \ref{S:LDPresults}, we recall the large deviations results of \cite{Spiliopoulos2014} that are used in this paper. We also review there the concept of importance sampling and that of classical subsolutions to related HJB equations. Section \ref{S:MainResult} contains the main result of this paper and its proof, this is Theorem \ref{T:UniformlyLogEfficient}. Section \ref{S:Examples} has an illustrating example and a simulation study to illustrate the main theoretical results of this paper.  In Appendix \ref{App:AuxiliaryTechnicalResults} the proof of a necessary technical result due to the randomness of the environment is given. 

\section{Notation and description of the random environment}\label{S:Notation}
In this section, we describe in detail the random environment and its properties. Also we set-up notation and present the main assumptions of the paper. The material of this section is classical in the random homogenization literature, e.g., \cite{JikovKozlovOleinik1994,Olla1994,KomorowskiLandimOlla2012}, and we review here the results that are being used in the later sections of the paper. We conclude this section with two examples.

Since we are dealing with random environments, we need to make certain assumptions on the random environment that will guarantee the necessary ergodic theorems that we need. Based on \cite{JikovKozlovOleinik1994}, we assume that there is a group of measure preserving transformation $\{\tau_{y}, y\in\mathbb{R}^{d-m}\}$ acting ergodically on $\Gamma$ that is defined as follows.
\begin{definition}
\label{Def:medium}
\begin{enumerate}
\item {$\tau_{y}$ preserves the measure, namely $\forall y\in\mathbb{R}^{d-m}$
and $\forall A\in\mathcal{G}$ we have $\nu(\tau_{y}A)=\nu(A)$.}

\item {The action of $\{\tau_{y}: y\in\mathbb{R}^{d-m}\}$ is ergodic, that is if
$A=\tau_{y}A$ for every $y\in\mathbb{R}^{d}$ then $\nu(A)=0$ or $1$.}

\item { For every measurable function $f$ on $\left(  \Gamma, \mathcal{G},
\nu\right)  $, the function $(y,\gamma)\mapsto f(\tau_{y}\gamma)$ is
measurable on $\left(  \mathbb{R}^{d-m}\times\Gamma, \mathbb{B}(\mathbb{R}%
^{d-m})\otimes\mathcal{G}\right)  $.}
\end{enumerate}
\end{definition}

Letting $\tilde{\phi}$ be a square integrable function in $\Gamma$, we define the operator $T_{y}\tilde{\phi}(\gamma)=\tilde{\phi}(\tau_{y}\gamma)$, which is a strongly continuous group of unitary maps in
$L^{2}(\Gamma)$, see \cite{Olla1994}. In addition, we shall denote by $D_{i}$ the infinitesimal generator  of $T_{y}$ in the direction
$i$, which is a closed and densely defined generator, see \cite{Olla1994}.

In order to guarantee that involved functions are ergodic and stationary random fields on $\mathbb{R}^{d-m}$,
for $\tilde{\phi}\in L^{2}(\Gamma)$, we define $\phi(y,\gamma)=\tilde{\phi}(\tau_{y}\gamma)$.  Similarly, for a measurable function $\tilde{\phi}:\mathbb{R}^{m}\times\Gamma\mapsto\mathbb{R}^{m}$
we consider the (locally) stationary random field $(x,y) \mapsto \tilde{\phi}(x,\tau_{y}\gamma)=\phi(x,y,\gamma)$. Then, it is guaranteed that $\phi(y,\gamma)$ (respectively $\phi(x,y,\gamma)$) is a stationary (respectively locally stationary) ergodic random field.

The coefficients, $b,c,\sigma,f,g,\tau_{1},\tau_{2}$ of (\ref{Eq:Main}) are defined through this procedure and therefore are guaranteed to be ergodic and stationary random fields. In particular, we start with $L^{2}(\Gamma)$ functions $\tilde{b}(\gamma),\tilde{c}(x,\gamma),\tilde{\sigma}(x,\gamma),\tilde{f}(\gamma),\tilde{g}(x,\gamma)$, $\tilde{\tau}_{1}(\gamma),\tilde{\tau}_{2}(\gamma)$ and we define the coefficients of (\ref{Eq:Main}) via the relations
$b(y,\gamma)=\tilde{b}(\tau_{y}\gamma),c(x,y,\gamma)=\tilde{c}(x,\tau_{y}\gamma),\sigma(x,y,\gamma)=\tilde{\sigma}(x,\tau_{y}\gamma),f(y,\gamma)=\tilde{f}(\tau_{y}\gamma),
g(x,y,\gamma)=\tilde{g}(x,\tau_{y}\gamma),\tau_{1}(y,\gamma)=\tilde{\tau}_{1}(\tau_{y}\gamma)$ and $\tau_{2}(y,\gamma)=\tilde{\tau}_{2}(\tau_{y}\gamma)$.

We need to make certain assumptions both on the coefficients of (\ref{Eq:Main}) and on the properties of the random environment. In regards to the coefficients of (\ref{Eq:Main}) we assume that
\begin{condition}
\label{A:Assumption1}

\begin{enumerate}
\item For every fixed $\gamma\in\Gamma$, the diffusion matrices $\sigma\sigma^{T}$ and $\tau_{1}\tau_{1}^{T}+\tau_{2}\tau_{2}^{T}$ are uniformly nondegenerate.
\item The functions $b(y,\gamma),c(x,y,\gamma),\sigma(x,y,\gamma), f(y,\gamma), g(x,y,\gamma), \tau_{1}(y,\gamma)$ and $\tau_{2}(y,\gamma)$ are
$C^{1}(\mathbb{R}^{d-m})$ in $y$ and $C^{1}(\mathbb{R}^{m})$ in $x$ with all
partial derivatives Lipschitz continuous and globally bounded in $x$ and $y $.
\end{enumerate}
\end{condition}

For every $\gamma\in\Gamma$, we define next the operator
\[
\mathcal{L}^{\gamma}=f(y,\gamma)\nabla_{y}\cdot+\text{\emph{tr}}\left[\left(
\tau_{1}(y,\gamma)\tau^{T}_{1}(y,\gamma)+\tau_{2}(y,\gamma)\tau^{T}_{2}(y,\gamma)\right)\nabla_{y}\nabla_{y}\cdot\right] \label{Def:OperatorFastProcess}
\]
and we let $Y_{t,\gamma}$ to be the corresponding Markov process.

As it is common in the random homogenization literature, \cite{PapanicolaouVaradhan1982,Osada1983,Olla1994}, we can associate the
canonical process on $\Gamma$ defined by the environment $\gamma$. This is  a Markov process on $\Gamma$, denoted by $\gamma_{t}$, with continuous
transition probability densities with respect to $d$-dimensional Lebesgue
measure, e.g., \cite{Olla1994}. To be precise, we define
\begin{align*}
\gamma_{t}  &  =\tau_{Y_{t,\gamma}}\gamma\label{Eq:EnvironmentProcess}\\
\gamma_{0}  &  =\tau_{y_{0}}\gamma\nonumber
\end{align*}

It is relatively straightforward to see that  the infinitesimal generator of the
Markov process $\gamma_{t}$ is given by
\[
\tilde{L}=\tilde{f}(\gamma)D\cdot+\text{\emph{tr}}\left[ \left( \tilde{\tau
}_{1}(\gamma)\tilde{\tau}_{1}^{T}(\gamma)+\tilde{\tau
}_{2}(\gamma)\tilde{\tau}_{2}^{T}(\gamma)\right)D^{2}\cdot\right].
\]

Next, we would like to have a closed form for the unique ergodic invariant measure for the environment process
$\{\gamma_{t}\}_{t\geq0}$. Such a closed form is useful for numerical feasibility purposes. For this reason,  we assume the following condition on the structure of the operator $\tilde{L}$, see \cite{Osada1983}.

\begin{condition}
\label{A:Assumption2} We can write the operator $\tilde{L}$ in the following generalized divergence form
\[
\tilde{L}=\frac{1}{\tilde{m}(\gamma)}\left[  \sum_{i,j}D_{i}\left(  \tilde
{a}\tilde{a}_{i,j}^{T}(\gamma)D_{j}\cdot\right)  +\sum_{j}\tilde{\beta}%
_{j}(\gamma)D_{j}\cdot\right]
\]
where $\tilde{\beta}_{j}=\tilde{m}\tilde{f}_{j}-\sum_{i}D_{i}\left(
\left(\tilde{\tau}_{1}\tilde{\tau}_{1}^{T}+\tilde{\tau}_{2}\tilde{\tau}_{2}^{T}\right)_{i,j}\tilde{m}\right)  $ and $\tilde{a}%
\tilde{a}_{i,j}^{T}=
\left(\tilde{\tau}_{1}\tilde{\tau}_{1}^{T}+\tilde{\tau}_{2}\tilde{\tau}_{2}^{T}\right)_{i,j}\tilde{m}$. We
assume that $\tilde{m}(\gamma)$ is bounded from below and from above with
probability $1$, that there exist smooth $\tilde{d}_{i,j}(\gamma)$ such that
$\tilde{\beta}_{j}=\sum_{j}D_{j}\tilde{d}_{i,j}$ with $|\tilde{d}_{i,j}|\leq
M$ for some $M<\infty$ almost surely and
\[
\text{div }\tilde{\beta}=0\text{ in distribution},\quad\text{i.e.,}%
\int_{\Gamma}\sum_{j=1}^{d}\tilde{\beta}_{j}(\gamma)D_{j}\tilde{\phi}%
(\gamma)\nu(d\gamma)=0,\quad\forall\tilde{\phi}\in\mathcal{H}^{1},
\]
where the Sobolev space $\mathcal{H}^{1}=\mathcal{H}^{1}(\nu)$ is the Hilbert space  equipped with the
inner product
\[
(\tilde{f},\tilde{g})_{1}=\sum_{i=1}^{d}(D_{i}\tilde{f},D_{i}\tilde{g}).
\]

\end{condition}

We will denote by $\mathbb{E}^{\nu}$  the expectation
operator with respect to the measure $\nu$. In particular, we have the following key result, Proposition \ref{P:NewMeasureRandomCase}.
\begin{proposition}
[\cite{Osada1983} and Theorem 2.1 in \cite{Olla1994}]%
\label{P:NewMeasureRandomCase} Assume Conditions \ref{A:Assumption1} and
\ref{A:Assumption2}. Define a measure on $(\Gamma,\mathcal{G})$ by%
\[
\pi(d\gamma)\doteq\frac{\tilde{m}(\gamma)}{\mathbb{E}^{\nu}\tilde{m}(\cdot
)}\nu(d\gamma).
\]
Then $\pi$ is the unique ergodic invariant measure for the environment process
$\{\gamma_{t}\}_{t\geq0}$.
\end{proposition}

We remark here that since $\tilde{m}$ is bounded from above and from below, $\mathcal{H}^{1}(\nu)$ and $\mathcal{H}^{1}(\pi)$ are equivalent.  The last tool that we need to introduce is the macroscopic problem, known as cell problem in the periodic homogenization literature or corrector in the homogenization literature in general. This is needed in the simulation algorithm only when $\tilde{b}\neq0$. Let us consider $\rho>0$ and consider the following problem on $\Gamma$
\begin{equation}
\rho\tilde{\chi}_{\rho}-\tilde{L}\tilde{\chi}_{\rho}=\tilde{b}.
\label{Eq:RandomCellProblem}%
\end{equation}

Let us assume that $\tilde{b}\in L^{2}(\nu)$ with $\left\Vert \tilde{b}\right\Vert_{-1}<\infty$, where $\left\Vert \cdot\right\Vert_{-1}$, is the natural norm of the dual space to $\mathcal{H}^{1}$ (see \cite{Olla1994, KomorowskiLandimOlla2012} for more details on  these spaces). By Lax-Milgram lemma, see \cite{Olla1994, KomorowskiLandimOlla2012},
equation (\ref{Eq:RandomCellProblem}) has a unique weak solution in the abstract Sobolev space
$\mathcal{H}^{1}$ or equivalently in $\mathcal{H}^{1}(\pi)$. At this point, we note that in the periodic case one also considers (\ref{Eq:RandomCellProblem}), but one can then take $\rho=0$ given that $b$ averages to zero when is integrated against the invariant measure $\pi$. However, in the random case, (\ref{Eq:RandomCellProblem}) with $\rho=0$ does not necessarily has a well defined solution (even if $b$ averages to zero when is integrated against the invariant measure $\pi$), see for example \cite{KomorowskiLandimOlla2012}. If there is special structure to the problem, as in the example of Section \ref{S:Examples}, one can solve or approximate the solution to (\ref{Eq:RandomCellProblem}) even with $\rho=0$. But in the general case, we consider the equation with $\rho>0$ and in the end, the homogenization theorem is proven by taking appropriate sequences $\rho=\rho(\epsilon)$ such that $\rho(\epsilon)\downarrow 0$ as $\epsilon\downarrow 0$. Taking $\rho\downarrow 0$ is allowed by the following well known properties of the solution to (\ref{Eq:RandomCellProblem}), (see \cite{Olla1994,Osada1983,PapanicolaouVaradhan1982}),
\begin{enumerate}
\item{There is a constant $K$ that is independent of $\rho$ such that
\[
\rho\mathbb{E}^{\pi}\left[  \tilde{\chi}_{\rho}(\cdot)\right]  ^{2}%
+\mathbb{E}^{\pi}\left[  D\tilde{\chi}_{\rho}(\cdot)\right]  ^{2}\leq K
\]}
\item{$\tilde{\chi}_{\rho}$ has an $\mathcal{H}^{1}$
strong limit, i.e., there exists a $\tilde{\chi}_{0}\in\mathcal{H}^{1}(\pi)$
such that
\[
\lim_{\rho\downarrow0}\left\Vert \tilde{\chi}_{\rho}(\cdot)-\tilde{\chi}%
_{0}(\cdot)\right\Vert _{1}=0\quad\text{ and }\quad \lim_{\rho\downarrow0}\rho\mathbb{E}^{\pi}\left[  \tilde{\chi}_{\rho}%
(\cdot)\right]  ^{2}=0.
\]
For notational convenience we set $\tilde{\xi}=D\tilde{\chi}_{0}$.}
\end{enumerate}
An important point coming our of Theorem \ref{T:UniformlyLogEfficient} is that the optimal control does not use $\chi_{\rho}$ itself, but its gradient $D\chi_{\rho}$.

We conclude this section with two representative examples. The first example, Example \ref{Ex:PeriodicCase}, indicates that in our setting the periodic case can be essentially viewed as a special case of the current set-up. The second example, Example \ref{Ex:RandomCase}, is a typical case where the operator of the fast process satisfies Condition \ref{A:Assumption2} and thus by Proposition \ref{P:NewMeasureRandomCase}, there is a closed form for the corresponding unique ergodic invariant measure.
\begin{example}\label{Ex:PeriodicCase}
In the periodic case, say with period $1$, $\Gamma$ is the unit
torus and $\nu$ is Lebesgue measure on $\Gamma$. Moreover, for every
$\gamma\in\Gamma$ we can consider the shift operators $\tau_{y}\gamma
=(y+\gamma)\mod 1$ and obtain $f(y,\gamma)=\tilde{f}(y+\gamma)$ for a periodic
function $\tilde{f}$ with period $1$. Hence, the periodic case is a special case of the case considered in this paper. Large deviations and related importance sampling schemes for the periodic case, were considered in \cite{Spiliopoulos2013}. Simulation results for closely related problems for diffusions with multiple scales in periodic environments can be found in \cite{DupuisSpiliopoulosWang,DupuisSpiliopoulosWang2}.
\end{example}

\begin{example}\label{Ex:RandomCase}
Let $\Gamma$ be the space of all
$C^{2}$ vector fields equipped with the Fr\'{e}chet metric. $\mathcal{G}$ is
the Borel $\sigma-$algebra, $\tau_{y}\gamma(\cdot)=\gamma(y+\cdot)$ for every
$(y,\gamma)\in\mathbb{R}^{d}\times\Gamma$ and $\nu$ is a Borel probability
measure that is invariant under the group of shifts $\tau_{y}$.

A particular example of fast dynamics of interest, that is studied in more detail in Section
\ref{S:Examples}, is the gradient case. In particular, let us choose $\tilde{f}(\gamma)=-D
\tilde{Q}(\gamma)$, where $\tilde{Q}$ is such that $f(y,\cdot)=\tilde{f}(\tau_{y}\cdot)$ is in $C^{1}(\mathbb{R}^{d-m})$, and $\tilde{\tau}_{1}(\gamma)=\sqrt{2D}\theta=\text{constant}$ and $\tilde{\tau}_{2}(\gamma)=\sqrt{2D}\sqrt{1-\theta^{2}}=\text{constant}$.

Then, it is easy to see  that
$\tilde{m}(\gamma)=\exp[-\tilde{Q}(\gamma)/D]$ defines the unique ergodic invariant measure for the environment process $\{\gamma_{t}\}_{t\geq 0}$ of Proposition \ref{P:NewMeasureRandomCase} and $\tilde{\beta}_{j}=0$
for all $1\leq j\leq d$. This example will be discussed in more details  in Section \ref{S:Examples} with simulation results as well.
\end{example}

\section{Related large deviations results and preliminaries on importance sampling}\label{S:LDPresults}

In this section we review the large deviations results of \cite{Spiliopoulos2014} that guide the design of provably-efficient importance sampling schemes. Then, we recall the general problem of constructing efficient importance sampling based on constructing appropriate subsolutions to the corresponding HJB equation.

\subsection{Related large deviations behavior}
As it is proven in \cite{Spiliopoulos2014}, see also Theorem \ref{T:MainTheorem3} below, the imposed assumptions guarantee that the family $\{X^{\epsilon,\gamma}, \epsilon>0\}$ satisfies a quenched large deviations principle, i.e., almost sure with respect to the random environment $\gamma\in \Gamma$. Let us denote the corresponding action functional by $S(\phi)$. At a heuristic level this means that given any deterministic path $\{\phi_{s}, s\in[t_{0},T]\}$, the probability that the process $\{X^{\epsilon,\gamma}_{s},s\in[t_{0},T]\}$ is in the small neighborhood of $\phi_{\cdot}$ is approximately $e^{-\frac{1}{\epsilon}S(\phi)}$. Large deviations theory is a well developed subject in probability theory and in particular for the case of small noise diffusion the interested reader is referred to the classical manuscript \cite{FW1}. In our case, the large deviations action functional is given in Theorem \ref{T:MainTheorem3}. But before mentioning the large deviations result, we briefly mention for completeness the homogenization theorem. We have

\begin{theorem}[Theorem 3.3 in \cite{Spiliopoulos2014}]\label{T:QuenchedLLN}
Let $\{\left(X^{\epsilon,\gamma}_{s},Y^{\epsilon,\gamma}_{s}\right),\epsilon>0, s\in[t_{0},T]\}$ be, for fixed $\gamma\in\Gamma$, the unique strong
solution to (\ref{Eq:Main}) with initial point $\left(X^{\epsilon,\gamma}_{t_{0}},Y^{\epsilon,\gamma}_{t_{0}}\right)=(x_{0},y_{0})$. Under Conditions
\ref{A:Assumption1} and \ref{A:Assumption2}, $\{X^{\epsilon,\gamma},\epsilon>0\}$
converges, almost surely with respect to $\gamma\in\Gamma$, to $X^{0}$, the (deterministic) solution of the differential equation
\begin{equation*}
X^{0}_{s}=x_{0}+\int_{t_{0}}^{s}r(X^{0}_{\kappa})d\kappa
\end{equation*}
where
\[
r(x)=\lim_{\rho\downarrow0}\mathbb{E}^{\pi}\left[  \tilde{c}(x,\cdot) +D\tilde{\chi}_{\rho
}(\cdot) \tilde{g}(x,\cdot)\right]  =\mathbb{E}^{\pi}[\tilde{c}(x,\cdot)+\tilde{\xi}%
(\cdot)\tilde{g}(x,\cdot)]
\]
and $\rho=\rho(\epsilon)=\frac{\delta^{2}}{\epsilon}$.
\end{theorem}

Let us denote by $\mathcal{AC}([t_{0},T];\mathbb{R}^{m})$ the set of absolutely continuous functions from $[t_{0},T]$ to $\mathbb{R}^{m}$.

\begin{theorem}[Theorem 3.5 in \cite{Spiliopoulos2014}]
\label{T:MainTheorem3} Let $\{\left(X^{\epsilon,\gamma},Y^{\epsilon,\gamma}\right),\epsilon>0\}$ be, for fixed $\gamma\in\Gamma$, the unique strong
solution to (\ref{Eq:Main}). Under Conditions
\ref{A:Assumption1} and \ref{A:Assumption2}, $\{X^{\epsilon,\gamma},\epsilon>0\}$
satisfies, almost surely with respect to $\gamma\in\Gamma$, the large deviations principle with rate function
\begin{equation*}
S_{t_{0}T}(\phi)=%
\begin{cases}
\frac{1}{2}\int_{t_{0}}^{T}(\dot{\phi}_{s}-r(\phi_{s}))^{T}q^{-1}(\phi_{s}%
)(\dot{\phi}_{s}-r(\phi_{s}))ds & \text{if }\phi\in\mathcal{AC}%
([t_{0},T];\mathbb{R}^{m}) \text{  and } \phi_{t_{0}}=x_{0}\\
+\infty & \text{otherwise.}%
\end{cases}
\label{Eq:ActionFunctional1}%
\end{equation*}
where $r(x)$ is as in Theorem \ref{T:QuenchedLLN} and using again $\rho=\rho(\epsilon)=\frac{\delta^{2}}{\epsilon}$,
\begin{align*}
q(x)&=\lim_{\rho\downarrow0}\mathbb{E}^{\pi}\left[  (\tilde{\sigma}(x,\cdot)+D\tilde{\chi}_{\rho}%
(\cdot)\tilde{\tau}_{1}(\cdot))(\tilde{\sigma}(x,\cdot)+D\tilde{\chi}_{\rho}%
(\cdot)\tilde{\tau}_{1}(\cdot))^{T}+\left(D\tilde{\chi}_{\rho}%
(\cdot)\tilde{\tau}_{2}(\cdot)\right)\left(D\tilde{\chi}_{\rho}%
(\cdot)\tilde{\tau}_{2}(\cdot)\right)^{T}\right]\nonumber\\
&=\mathbb{E}^{\pi}\left[  (\tilde{\sigma}(x,\cdot)+\tilde{\xi}(\cdot)\tilde{\tau}_{1}(\cdot))(\tilde{\sigma}(x,\cdot)+\tilde{\xi}
(\cdot)\tilde{\tau}_{1}(\cdot))^{T}+\left(\tilde{\xi}
(\cdot)\tilde{\tau}_{2}(\cdot)\right)\left(\tilde{\xi}
(\cdot)\tilde{\tau}_{2}(\cdot)\right)^{T}\right]
\end{align*}
\end{theorem}

It is clear that the coefficients $r(x)$ and $q(x)$ are obtained by homogenizing (\ref{Eq:Main}) by taking $\delta\downarrow 0$ with $\epsilon$ fixed. The  form of the action functional can be recognized as the one that would come up when considering large deviations for  the homogenized system. This is also implied by the fact that $\delta$ goes to zero faster than $\epsilon$, since  $\epsilon/\delta\uparrow\infty$.

\begin{remark}
It is clear that if $b=0$, then $\chi_{\rho}=0$ and then $r(x),q(x)$ take the simplified form $r(x)=\mathbb{E}^{\pi}[\tilde{c}(x,\cdot)]$ and $q(x)=\mathbb{E}^{\pi}\left[  \tilde{\sigma}(x,\cdot) \tilde{\sigma}(x,\cdot)^{T}\right]$.
\end{remark}

In particular, Theorem \ref{T:MainTheorem3} implies the following. Let us consider $h:\mathbb{R}^{m}\mapsto\mathbb{R}$ to be a continuous and bounded function and assume that we are interested in computing
\[
\theta(\epsilon,\gamma)=\mathbb{E}\left[e^{-\frac{1}{\epsilon}h(X^{\epsilon,\gamma}_{T})}\Big | X^{\epsilon,\gamma}_{t_{0}}=x_{0}, Y^{\epsilon,\gamma}_{t_{0}}=y_{0}\right]
\]
for small $\epsilon>0$. Then, by Theorem \ref{T:MainTheorem3},  for almost every realization of the random environment $\gamma\in \Gamma$, we have that
\[
\lim_{\epsilon\downarrow 0}-\epsilon\log \theta(\epsilon,\gamma)=G(t_{0},x_{0})
\]
where
\[
G(t_{0},x_{0})=\inf_{\phi\in\mathcal{C}([t_{0},T]):\phi_{t_{0}}=x_{0}}\left\{S_{t_{0}T}(\phi)+h(\phi_{T})\right\}.
\]

Note that the limit is independent of the initial point $y$ of the fast component $Y^{\epsilon,\gamma}$.

\subsection{Generalities on importance sampling} Let us define by $\Delta^{\epsilon,\gamma}(t_{0},x_{0},y_{0})$ an unbiased estimator of $\theta(\epsilon,\gamma)$, defined on some probability space $\bar{\mathbb{P}}$ with corresponding expectation operator $\bar{\mathbb{E}}$. To be precise, we have
\[
\bar{\mathbb{E}}\Delta^{\epsilon,\gamma}(t_{0},x_{0},y_{0})=\theta(\epsilon,\gamma).
\]

As it is well known, in order to estimate $\theta(\epsilon,\gamma)$ via Monte-Carlo, one generates many independent copies of $\Delta^{\epsilon,\gamma}(t_{0},x_{0},y_{0})$ and the sample mean is the estimate.

The precise number of independent copies is related to the overall desired accuracy and is measured by the variance of the estimator. Hence, an efficient estimator is the one that has minimum variance and due to unbiasedness, minimizing the variance is the same as minimizing the second moment.  Jensen's inequality guarantees that
\[
\mathbb{E}\left[\Delta^{\epsilon,\gamma}(t_{0},x_{0},y_{0})\right]^{2}\geq \left(\bar{\mathbb{E}}\Delta^{\epsilon,\gamma}(t_{0},x_{0},y_{0})\right)^{2}=\theta(\epsilon,\gamma)^{2}
\]

On the other hand, by the large deviations theorem, we have that for almost every $\gamma\in\Gamma$
\[
\lim_{\epsilon\downarrow 0}-\epsilon\log\bar{\mathbb{E}}\left[\Delta^{\epsilon,\gamma}(t_{0},x_{0},y_{0})\right]^{2}\leq 2G(t_{0},x_{0})
\]

Therefore, if we actually have the opposite inequality as well, i.e., if for almost every $\gamma\in\Gamma$
 \[
\lim_{\epsilon\downarrow 0}-\epsilon\log\bar{\mathbb{E}}\left[\Delta^{\epsilon,\gamma}(t_{0},x_{0},y_{0})\right]^{2}\geq 2G(t_{0},x_{0}),
\]
then the estimator $\Delta^{\epsilon,\gamma}(t_{0},x_{0},y_{0})$ is called \textit{logarithmic asymptotically optimal}.

Let us next elaborate on the construction of the importance sampling scheme. For notational convenience, we define the  $2\kappa-$dimensional Wiener process $Z_{\cdot}=\left(W_{\cdot}, B_{\cdot}\right)$.
Let $u(t,x,y)$ be a control that is sufficiently smooth and bounded and may also depend on $\gamma\in \Gamma$. The family of probability measures $\bar{\mathbb{P}}$ is defined via the change of measure
\[
\frac{d\bar{\mathbb{P}}}{d\mathbb{P}}=\exp\left\{  -\frac
{1}{2\epsilon}\int_{t_{0}}^{T}\left\Vert u(s,X^{\epsilon,\gamma}_{s},Y^{\epsilon,\gamma}_{s})\right\Vert
^{2}ds+\frac{1}{\sqrt{\epsilon}}\int_{t_{0}}^{T}\left\langle u(s,X^{\epsilon,\gamma}_{s},Y^{\epsilon,\gamma}_{s}),dZ_{s}\right\rangle \right\}  .
\]
Then, Girsanov's Theorem guarantees that
\[
\bar{Z}_{s}=Z_{s}-\frac{1}{\sqrt{\epsilon}}\int_{t_{0}}^{s}u(\kappa,X^{\epsilon,\gamma
}_{\kappa},Y^{\epsilon,\gamma}_{\kappa})d\kappa,~~~s\in[t_{0},T]
\]
is a Wiener process on $[t_{0},T]$ under the probability measure $\bar
{\mathbb{P}}$. Moreover, under $\bar{\mathbb{P}}$ and $(\bar{X}^{\epsilon},\bar{Y}^{\epsilon})=(\bar{X}^{\epsilon,\gamma},\bar{Y}^{\epsilon,\gamma})$
satisfies $\bar{X}^{\epsilon}_{t_{0}}=x_{0}$, $\bar{Y}^{\epsilon}_{t_{0}}=y_{0}$ and for
$s\in(t_{0},T]$ it is the unique strong solution of
\begin{eqnarray}
d\bar{X}^{\epsilon}_{s}&=&\left[  \frac{\epsilon}{\delta}b\left(\bar{Y}^{\epsilon}_{s},\gamma\right)+c\left(  \bar{X}^{\epsilon}_{s}%
,\bar{Y}^{\epsilon}_{s},\gamma\right)+\sigma\left(  \bar{X}_{s}^{\epsilon},\bar{Y}_{s}^{\epsilon},\gamma\right)  u_{1}(s)\right]   dt+\sqrt{\epsilon}%
\sigma\left(  \bar{X}^{\epsilon}_{s},\bar{Y}^{\epsilon}_{s},\gamma\right)
d\bar{W}_{s}, \nonumber\\
d\bar{Y}^{\epsilon}_{s}&=&\frac{1}{\delta}\left[  \frac{\epsilon}{\delta}f\left(\bar{Y}^{\epsilon}_{s},\gamma\right)  +g\left(  \bar{X}^{\epsilon}_{s}
,\bar{Y}^{\epsilon}_{s},\gamma\right)+\tau_{1}\left(\bar{Y}^{\epsilon}_{s},\gamma\right)u_{1}(s)+
\tau_{2}\left(\bar{Y}^{\epsilon}_{s},\gamma\right)u_{2}(s)\right]   dt\nonumber\\
& &\hspace{5cm}+\frac{\sqrt{\epsilon}}{\delta}\left[
\tau_{1}\left(\bar{Y}^{\epsilon}_{s},\gamma\right)
d\bar{W}_{s}+\tau_{2}\left(\bar{Y}^{\epsilon}_{s},\gamma\right)d\bar{B}_{s}\right],\label{Eq:Main2}\\
\bar{X}^{\epsilon}_{t_{0}}&=&x_{0},\hspace{0.2cm}\bar{Y}^{\epsilon}_{t_{0}}=y_{0}\nonumber
\end{eqnarray}

where $(u_{1}(s), u_{2}(s))$ denote the first and second component of the control
\[
u(s,\bar{X}^{\epsilon}_{s},\bar{Y}^{\epsilon}_{s})=(u_{1}(s,\bar{X}^{\epsilon}_{s},\bar{Y}^{\epsilon}_{s}), u_{2}(s,\bar{X}^{\epsilon}_{s},\bar{Y}^{\epsilon}_{s})).
\]

Then, under $\bar{\mathbb{P}}$
\[
\Delta^{\epsilon,\gamma}(t_{0},x_{0},y_{0})=\exp\left\{  -\frac{1}{\epsilon}h(\bar{X}^{\epsilon
}_{T})\right\}  \frac{d\mathbb{P}}{d\bar{\mathbb{P}}}(\bar{X}^{\epsilon
}, \bar{Y}^{\epsilon}),
\]
is an unbiased estimator for $\theta(\epsilon,\gamma)$.  The performance of
this estimator is characterized by the decay rate of its second moment
\begin{equation}
Q^{\epsilon,\gamma}(t_{0},x_{0},y_{0};u)\doteq\bar{\mathbb{E}}\left[  \exp\left\{
-\frac{2}{\epsilon}h(\bar{X}^{\epsilon,\gamma}_{T})\right\}  \left(  \frac{d\mathbb{P}%
}{d\bar{\mathbb{P}}}(\bar{X}^{\epsilon,\gamma}, \bar{Y}^{\epsilon,\gamma})\right)  ^{2}\right].
\label{Eq:2ndMoment1}%
\end{equation}

Construction of asymptotically optimal importance sampling schemes is done by appropriately choosing the control $u$ in (\ref{Eq:2ndMoment1}). The goal is to be able to control
the behavior of the second moment $Q^{\epsilon,\gamma}(t_{0},x_{0},y_{0};u)$. As we shall see in Theorem \ref{T:UniformlyLogEfficient} the construction of controls that lead to asymptotically optimal changes of measures is based on the following two ingredients
\begin{enumerate}
\item{The solution to the macroscopic problem (\ref{Eq:RandomCellProblem}), which needs to be taken into account due to the homogenization effects, when $b\neq 0$.}
\item{The gradient of a subsolution to the PDE that the function $G(t,x)$
 satisfies. }
\end{enumerate}

The first ingredient is related to  homogenization theory for the related HJB equation. Taking into account the macroscopic problem turns out to be crucial in establishing efficiency of the importance sampling change of measure when $b\neq 0$. This point has been extensively investigated in  \cite{DupuisSpiliopoulosWang} for the periodic case. It has been shown there that omitting the corrector does not lead to asymptotic efficient schemes and actually leads to schemes with performance comparable to standard Monte Carlo, i.e., to no change of measure.  Clearly, in the more complex case where the multiscale environment is random, one does not expect something different.

The second ingredient is related to the gradient of the solution to the associated homogenized HJB equation. As we shall see in Theorem \ref{T:UniformlyLogEfficient}, the gradient of the solution to the associated  homogenized HJB equation enters in the construction of the optimal change of measure. However,  for reasons that have to do with either feasibility in computing the gradient or lack of smoothness of the solution to the HJB equation, one looks for appropriate subsolutions to the associated  homogenized HJB equation. For this reason,  let us recall the notion of a subsolution to the appropriate HJB equation. It is known that in general $G(t,x)$ is the viscosity solution to the HJB equation, see for example \cite{FlemingSoner2006},
\begin{align}
\partial_{s}G(s,x)+H(x,\nabla_{x}G(s,x))=0,\qquad G(T,x)=h(x)\label{Eq:HJBequation2}
\end{align}
where the Hamiltonian is
\[
H(x,p)=\left<r(x), p\right>-\frac{1}{2}\left\Vert q^{1/2}(x)p\right\Vert^{2}
\]
 with $r(x),q(x)$ the coefficients defined in Theorems \ref{T:QuenchedLLN} and \ref{T:MainTheorem3}. If optimal or nearly optimal schemes are overly complicated and difficult to implement, then one may choose to construct sub-optimal but simpler schemes, but with guaranteed performance. Rare events associated with multiscale problems are rather complicated
and many times is it  very difficult to construct asymptotically optimal schemes. One way to circumvent this difficulty is by constructing appropriate sub-optimal schemes
with precise bounds on asymptotic performance. One way that this is made possible is via the subsolution approach, introduced in \cite{DupuisWang2}. Let us first recall the definition of a subsolution.

\begin{definition}
\label{Def:ClassicalSubsolution} A function
$\bar{U}(s,x):[t_{0},T]\times \mathbb{R}^{m}\mapsto\mathbb{R}$ is a
classical subsolution to the HJB equation (\ref{Eq:HJBequation2}) if

\begin{enumerate}
\item $\bar{U}$ is continuously differentiable,

\item $\partial_{s}\bar{U}(s,x)+H(x,\nabla_{x}\bar{U}(s,x))\geq0$ for every
$(s,x)\in(t_{0},T)\times\mathbb{R}^{m}$,

\item $\bar{U}(T,x)\leq h(x)$ for $x\in\mathbb{R}^{m}$.
\end{enumerate}
\end{definition}

For illustration purposes and in order to avoid several technical problems, we will
 impose stronger regularity conditions on the subsolutions to
be considered than those of Definition
\ref{Def:ClassicalSubsolution}.  Nevertheless,  we would like to
mention that the uniform bounds that will be assumed in Condition
\ref{Cond:ExtraReg} can be replaced by milder conditions with the
expense of working harder to establish the results.

\begin{condition}
\label{Cond:ExtraReg} There exists a subsolution $\bar{U}$ which has continuous derivatives up to
order $1$ in $t$ and order $2$ in $x$, and the first and second
derivatives in $x$ are uniformly bounded.
\end{condition}

\section{Statement and proof of the main result}\label{S:MainResult}

In this section, we state and prove our main result on asymptotically optimal changes of measure for small noise diffusions in random environments. Our main theorem is as follows.

\begin{theorem}
\label{T:UniformlyLogEfficient} Let $\{\left(X^{\epsilon}_{s}, Y^{\epsilon}_{s} \right),\epsilon>0\}$ be
the solution to (\ref{Eq:Main}) for $s\in[t_{0},T]$ with initial point $(x_{0},y_{0})$ at time $t_{0}$.  Consider a non-negative, bounded and continuous
function $h:\mathbb{R}^{m}\mapsto\mathbb{R}$ and let $\bar{U}(s,x)$ be a subsolution to the associated HJB equation. Assume Conditions
\ref{A:Assumption1}, \ref{A:Assumption2} and \ref{Cond:ExtraReg}. In the general case where $b\neq 0$, consider $\rho>0$ and define the (random) feedback control $u_{\rho}(s,x,y,\gamma)=\left(u_{1,\rho}(s,x,y,\gamma), u_{2,\rho}(s,x,y,\gamma)\right)$ by

 \begin{equation*}
u_{\rho}(s,x,y,\gamma)=\left(-\left(\sigma+D\chi_{\rho} \tau_{1}\right)^{T}(x,y,\gamma)\nabla_{x}\bar{U}(s,x), -\left(D\chi_{\rho}\tau_{2}\right)^{T}(y,\gamma)\nabla_{x}\bar{U}(s,x)\right)\label{Eq:feedback_controlReg1}
\end{equation*}

Then for $\rho=\rho(\epsilon)=\frac{\delta^{2}}{\epsilon}\downarrow 0$ we have that almost surely in $\gamma\in\Gamma$
\begin{equation}
\liminf_{\epsilon\rightarrow0}-\epsilon\ln Q^{\epsilon,\gamma}(t_{0},x_{0},y_{0};u_{\rho}(\cdot))\geq
G(t_{0},x_{0})+\bar{U}(t_{0},x_{0}). \label{Eq:GoalSubsolution}%
\end{equation}

If $b=0$, then set $u(s,x,y,\gamma)=\left(-\sigma^{T}(x,y,\gamma)\nabla_{x}\bar{U}(s,x), 0\right)$ and (\ref{Eq:GoalSubsolution}) holds with $u_{\rho}(\cdot)=u(\cdot)$.
\end{theorem}

Based on Theorem  \ref{T:UniformlyLogEfficient}, we can then establish Proposition \ref{P:UniformlyLogEfficient1}, which is about estimating probabilities of the form $\mathbb{P}_{t_{0},x_{0},y_{0}}[X^{\epsilon,\gamma}_{T}\in A]$. Even though this case cannot be recast as the expectation $\mathbb{E}\left[e^{-\frac{1}{\epsilon}h}\right]$ for some continuous function $h$, an approximating argument analogous to \cite{DupuisWang2} can be used to establish the claim. The details of the proof are omitted.
\begin{proposition}
\label{P:UniformlyLogEfficient1} Assume Conditions
\ref{A:Assumption1}, \ref{A:Assumption2} and \ref{Cond:ExtraReg}. Let $\{\left(X^{\epsilon}, Y^{\epsilon} \right),\epsilon>0\}$ be
the solution to (\ref{Eq:Main}) with initial point $(t_{0},x_{0},y_{0})$. Let $A\subset\mathbb{R}^{m}$ be a regular set with respect to the action functional $S$ and the initial point $(t_{0},x_{0},y_{0})$, i.e., the infimum of $S$ over the closure $\bar{A}$ is the same as
the infimum over the interior $A^{o}$.
Let
\[
h(x)=%
\begin{cases}
0 & \text{if }x\in A\\
+\infty & \text{if }x\notin A.
\end{cases}
\]
Let $u_{\rho}(s,x,y,\gamma)$  be defined as in Theorem \ref{T:UniformlyLogEfficient}. Then, for  $\rho=\rho(\epsilon)=\frac{\delta^{2}}{\epsilon}\downarrow 0$ we have that almost surely in $\gamma\in\Gamma$, (\ref{Eq:GoalSubsolution}) holds.
\end{proposition}

The following remark demonstrates how subsolutions quantify performance.
\begin{remark}\label{R:SubsolutionPerformance}
The subsolution property of $\bar{U}$ implies  that $0\leq \bar{U}(s,x)\leq
G(s,x)$ everywhere. By (\ref{Eq:GoalSubsolution}) this implies that  the scheme is logarithmic asymptotically optimal if $\bar
{U}(t_{0},x_{0})=G(t_{0},x_{0})$ at the starting point $(t_{0},x_{0})$. Standard Monte Carlo corresponds to choosing the subsolution $\bar{U}=0$. The latter imply that any subsolution scheme with
$$0\ll\bar{U}(t_{0},x_{0})\leq G(t_{0},x_{0}) $$
will outperform standard Monte Carlo measured by how close to $G$ the value of $\bar{U}$  at the initial point $(t_{0},x_{0})$ is.
\end{remark}

\begin{remark}\label{R:MetastabilityEffect}
Even though we will not elaborate more on this issue in this paper, we mention for completeness that one can merge the results of \cite{DupuisSpiliopoulosZhou2013, DupuisSpiliopoulos2014, Spiliopoulos2014a} with the ones of this paper. In particular, for multiscale problems with metastable features, one can use the change of measure given in Theorem \ref{T:UniformlyLogEfficient} and use as subsolutions $\bar{U}(s,x)$, the subsolutions constructed in  \cite{DupuisSpiliopoulosZhou2013, DupuisSpiliopoulos2014} that guarantee good performance even non-asymptotically. This is an important point that we plan to study in detail in a future work. The focus of this paper is the effect of the multiscale aspect.
\end{remark}

\begin{remark}\label{R:NoUnboundedDriftTerm}
It is clear that if $b=0$ then $\chi_{\rho}=0$, in which case the control leading to an asymptotically optimal change of measure simplifies to
\begin{equation*}
u(s,x,y,\gamma)=\left(-\sigma^{T}(x,y,\gamma)\nabla_{x}\bar{U}(s,x), 0\right),\label{Eq:feedback_controlReg2a}
\end{equation*}
as the statement of Theorem \ref{T:UniformlyLogEfficient} says.
\end{remark}
\begin{proof}[Proof of Theorem \ref{T:UniformlyLogEfficient}]

Let $\rho=\rho(\epsilon)>0$ such that $\rho(\epsilon)=\frac{\delta^{2}}{\epsilon}\downarrow0$ as $\epsilon\downarrow 0$. Recall the definition
 \begin{equation*}
u_{1,\rho}(s,x,y,\gamma)=-\left(\sigma+D\chi_{\rho} \tau_{1}\right)^{T}(x,y,\gamma)\nabla_{x}\bar{U}(s,x), \qquad u_{2,\rho}(s,x,y,\gamma)=-\left(D\chi_{\rho}\tau_{2}\right)^{T}(y,\gamma)\nabla_{x}\bar{U}(s,x)\label{Eq:feedback_controlReg2}
\end{equation*}
and set
\begin{eqnarray}
\bar{c}_{\rho}\left(  s,x,y,\gamma\right)&=&c(x,y,\gamma)-\sigma(x,y,\gamma)u_{1,\rho}(s,x,y,\gamma)\nonumber\\
\bar{g}_{\rho}\left(  s,x,y,\gamma\right)&=&g(x,y,\gamma)-\tau_{1}(y,\gamma)u_{1,\rho}(s,x,y,\gamma)-\tau_{2}(y,\gamma)u_{2,\rho}(s,x,y,\gamma).\nonumber
\end{eqnarray}

Let $(\hat{X}^{\epsilon},\hat{Y}^{\epsilon})=(\hat{X}^{\epsilon,\gamma},\hat{Y}^{\epsilon,\gamma})$ satisfying the initial condition
 $\hat{X}^{\epsilon}_{t_{0}}=x_{0}$, $\hat{Y}^{\epsilon}_{t_{0}}=y_{0}$ and for
$s\in(t_{0},T]$ being the unique strong solution of
\begin{eqnarray}
d\hat{X}^{\epsilon}_{s}&=&\left[  \frac{\epsilon}{\delta}b\left(\hat{Y}^{\epsilon}_{s},\gamma\right)+\bar{c}_{\rho}\left(  \hat{X}^{\epsilon}_{s}%
,\hat{Y}^{\epsilon}_{s},\gamma\right)+\sigma\left(  \hat{X}_{s}^{\epsilon},\hat{Y}_{s}^{\epsilon},\gamma\right) v_{1}(s)\right]   dt+\sqrt{\epsilon}%
\sigma\left(  \hat{X}^{\epsilon}_{s},\hat{Y}^{\epsilon}_{s},\gamma\right)
dW_{s}, \nonumber\\
d\hat{Y}^{\epsilon}_{s}&=&\frac{1}{\delta}\left[  \frac{\epsilon}{\delta}f\left(\hat{Y}^{\epsilon}_{s},\gamma\right)  +\bar{g}_{\rho}\left(  \hat{X}^{\epsilon}_{s}
,\hat{Y}^{\epsilon}_{s},\gamma\right)+\tau_{1}\left(\hat{Y}^{\epsilon}_{s},\gamma\right)v_{1}(s)+
\tau_{2}\left(\hat{Y}^{\epsilon}_{s},\gamma\right)v_{2}(s)\right]   dt\nonumber\\
& &\hspace{5cm}+\frac{\sqrt{\epsilon}}{\delta}\left[
\tau_{1}\left(\hat{Y}^{\epsilon}_{s},\gamma\right)
dW_{s}+\tau_{2}\left(\hat{Y}^{\epsilon}_{s},\gamma\right)dB_{s}\right],\label{Eq:Main3}\\
\hat{X}^{\epsilon}_{t_{0}}&=&x_{0},\hspace{0.2cm}\hat{Y}^{\epsilon}_{t_{0}}=y_{0}\nonumber
\end{eqnarray}
where $v(\cdot)=(v_{1}(\cdot),v_{2}(\cdot))\in\mathcal{A}$, the set of all $\mathfrak{F}-$progressively measurable $m$-dimensional processes satisfying $\mathbb{E}\int_{0}^{T}\left\Vert v(s)\right\Vert^{2}ds<\infty$.
It is clear that $(\hat{X}^{\epsilon},\hat{Y}^{\epsilon})=(\hat{X}^{\epsilon,\gamma},\hat{Y}^{\epsilon,\gamma})$ depends on the realization $\gamma$ of the random environment, but this is not explicitly denoted in the notation.

After these definitions and based on Condition \ref{A:Assumption1}, Lemma 4.3 of \cite{DupuisSpiliopoulosWang} guarantees the following representation of the second moment
\begin{align}
-\epsilon\log Q^{\epsilon,\gamma}(t_{0},x_{0},y_{0};u)  =\inf_{v\in\mathcal{A}}\mathbb{E}\left[  \frac{1}{2}\int_{t_{0}}^{T}\left\Vert
v(s)\right\Vert ^{2}ds-\int_{t_{0}}^{T}\Vert u_{\rho}(s,\hat{X}^{\epsilon}_{s},\hat{Y}^{\epsilon}
_{s},\gamma)\Vert^{2}ds+2h(\hat{X}^{\epsilon}_{T})\right] \label{Eq:ToBeBounded}
\end{align}

Next, we need to take the limit as $\epsilon\downarrow0$ of (\ref{Eq:ToBeBounded}). At this point we use Lemma \ref{L:RestrictionControls} stating that for $\delta,\epsilon$ small enough such that $\delta/\epsilon\ll 1$ the infimum in the representation (\ref{Eq:ToBeBounded}) can be taken over all controls $v\in\mathcal{A}$ satisfying the constraint
\begin{equation}
\mathbb{E}\int_{0}^{\frac{\delta^{2}T}{\epsilon}}\left\Vert v(s)\right\Vert^{2}ds\leq C\delta\label{Eq:RestrictionControls}
\end{equation}
where the constant $C$ depends on $T$, but not on $\delta,\epsilon$ or $\gamma$. This bound allows us to apply the large deviations results and ergodic theorems for controlled diffusions of the form (\ref{Eq:Main3}) in random environments of \cite{Spiliopoulos2014}. In turn this allows us to identify the limit infimum of (\ref{Eq:ToBeBounded}). In particular, let us define
\[
r_{\rho}(x)=\mathbb{E}^{\pi}\left[  \tilde{c}(x,\cdot) +D\tilde{\chi}_{\rho}(\cdot) \tilde{g}(x,\cdot)\right],
\]

\begin{align*}
q_{\rho}(x)&=\mathbb{E}^{\pi}\left[  (\tilde{\sigma}(x,\cdot)+D\tilde{\chi}_{\rho}%
(\cdot)\tilde{\tau}_{1}(\cdot))(\tilde{\sigma}(x,\cdot)+D\tilde{\chi}_{\rho}%
(\cdot)\tilde{\tau}_{1}(\cdot))^{T}+\left(D\tilde{\chi}_{\rho}%
(\cdot)\tilde{\tau}_{2}(\cdot)\right)\left(D\tilde{\chi}_{\rho}%
(\cdot)\tilde{\tau}_{2}(\cdot)\right)^{T}\right],
\end{align*}

\begin{align*}
\bar{r}_{\rho}(s,x)&=r_{\rho}(x)-\mathbb{E}^{\pi}\left[\left(\tilde{\sigma} \tilde{u}_{1,\rho}+D \tilde{\chi}_{\rho}\tilde{\tau}_{1}\tilde{u}_{1,\rho}+D \tilde{\chi}_{\rho}\tilde{\tau}_{2}\tilde{u}_{2,\rho} \right)(s,x,\cdot)\right]\nonumber\\
&=r_{\rho}(x)+q_{\rho}(x)\nabla_{x}\bar{U}(s,x)
\end{align*}
and
\[
\bar{L}_{\rho}(s,x,\beta)=\frac{1}{2}\left(\beta-\bar{r}_{\rho}(s,x)\right)q^{-1}_{\rho}(x)\left(\beta-\bar{r}_{\rho}(s,x)\right)^{T}
\]
we obtain that almost surely in $\gamma\in\Gamma$
\begin{align}
\liminf_{\epsilon\rightarrow0}-\epsilon\log Q^{\epsilon,\gamma}(t_{0},x_{0},y_{0};u)& \geq \inf_{\phi,\phi_{t_{0}}=x_{0}}\lim_{\rho\downarrow 0}\left[ \int_{t_{0}}^{T}\bar{L}_{\rho}(s,\phi_{s},\dot{\phi}_{s})ds\right.\nonumber\\
& \left.-\mathbb{E}^{\pi}\int_{t_{0}}^{T}\left[\left\Vert u_{1,\rho}(s,\phi_{s},\cdot,\gamma)\right\Vert ^{2}+\left\Vert u_{2,\rho}(s,\phi_{s},\cdot,\gamma)\right\Vert ^{2}\right]ds+2h(\phi_{T})\right].\label{Eq:ToBeBoundedAfterLimit}
\end{align}

In the last display, the first limiting term, i.e., $\int_{t_{0}}^{T}\bar{L}_{\rho}(s,\phi_{s},\dot{\phi}_{s})ds$, is obtained via the large deviations result of Theorem \ref{T:MainTheorem3}. As it is seen in the large deviations computations of \cite{Spiliopoulos2014}, the choice $\rho(\epsilon)=\frac{\delta^{2}}{\epsilon}$ allows to prove the large deviations upper and lower bounds, i.e., Theorem \ref{T:MainTheorem3}. The limit of the second term is due to the quenched ergodic theorem for controlled random diffusion processes given by (\ref{Eq:Main3}) with the control $v$ being such that (\ref{Eq:RestrictionControls}) holds; this is Lemma A.6 in \cite{Spiliopoulos2014}.

The next step is to appropriately combine the terms on the right hand side of (\ref{Eq:ToBeBoundedAfterLimit}) by recalling the definition of $u_{1,\rho}$ and $u_{2,\rho}$.  In particular using their definition we immediately obtain
\begin{align}
\mathbb{E}^{\pi}\int_{t_{0}}^{T}\left[\left\Vert u_{1,\rho}(s,\phi_{s},\cdot,\gamma)\right\Vert ^{2}+\left\Vert u_{2,\rho}(s,\phi_{s},\cdot,\gamma)\right\Vert ^{2}\right]ds=\int_{t_{0}}^{T}\langle\nabla_{x}\bar{U}(s,\phi_{s}),q_{\rho}(\phi_{s})\nabla
_{x}\bar{U}(s,\phi(s))\rangle ds\label{Eq:LimitingTerm1}
\end{align}
where we have used the definition of $q_{\rho}(x)$. Moreover, setting
\[
L_{\rho}(x,\beta)=\frac{1}{2}\left(\beta-r_{\rho}(x)\right)q^{-1}_{\rho}(x)\left(\beta-r_{\rho}(x)\right)^{T}
\]
we also have that
\begin{align}
\int_{t_{0}}^{T}\bar{L}_{\rho}(s,\phi_{s},\dot{\phi}_{s})ds&=\frac{1}{2}\int_{t_{0}}^{T}\left(\dot{\phi}_{s}-\bar{r}_{\rho}(s,\phi_{s})\right)q^{-1}_{\rho}(\phi_{s})\left(\dot{\phi}_{s}-\bar{r}_{\rho}(s,\phi_{s})\right)^{T}ds\nonumber\\
&=\frac{1}{2}\int_{t_{0}}^{T}\left[\left< \dot{\phi}_{s}-r_{\rho}(\phi_{s}),q^{-1}_{\rho}(\phi_{s})\left(\dot{\phi}_{s}-r_{\rho}(\phi_{s})\right)\right>
-2\langle\dot{\phi}_{s}-r_{\rho}(\phi
_{s}),\nabla_{x}\bar{U}(s,\phi_{s})\rangle\right.\nonumber\\
&\qquad\left.+\langle \nabla_{x}\bar{U}(s,\phi_{s}),q_{\rho}(\phi_{s})\nabla_{x}\bar{U}(s,\phi_{s})\rangle\right] ds\nonumber\\
&=\frac{1}{2}\int_{t_{0}}^{T}\left[L_{\rho}(\phi_{s},\dot{\phi}_{s})-2\langle\dot{\phi}_{s}-r_{\rho}(\phi
_{s}),\nabla_{x}\bar{U}(s,\phi_{s})\rangle\right.\nonumber\\
&\qquad\left.+\langle \nabla_{x}\bar{U}(s,\phi_{s}),q_{\rho}(\phi_{s})\nabla_{x}\bar{U}(s,\phi_{s})\rangle\right] ds\label{Eq:LimitingTerm2}
\end{align}

Combining (\ref{Eq:LimitingTerm1}) and (\ref{Eq:LimitingTerm2}), we subsequently have that
\begin{align}
&\liminf_{\epsilon\rightarrow0}-\epsilon\log Q^{\epsilon,\gamma}(t_{0},x_{0},y_{0};u) \geq \inf_{\phi,\phi_{t_{0}}=x_{0}}\lim_{\rho\downarrow 0}\left[ \int_{t_{0}}^{T}L_{\rho}(\phi_{s},\dot{\phi}_{s})ds +2h(\phi_{T})\right.\nonumber\\
&\quad \left.- \int_{t_{0}}^{T}\left(\langle\dot{\phi}_{s}-r_{\rho}(\phi
_{s}),\nabla_{x}\bar{U}(s,\phi_{s})\rangle+\frac{1}{2}\langle\nabla_{x}\bar{U}(s,\phi_{s}),q_{\rho}(\phi_{s})\nabla
_{x}\bar{U}(s,\phi_{s})\rangle\right) ds\right]\nonumber\\
&\geq \inf_{\phi,\phi_{t_{0}}=x_{0}}\left[ \int_{t_{0}}^{T}L(\phi_{s},\dot{\phi}_{s})ds +2h(\phi_{T})\right.\nonumber\\
&\quad \left.- \int_{t_{0}}^{T}\left(\langle\dot{\phi}_{s}-r(\phi
_{s}),\nabla_{x}\bar{U}(s,\phi_{s})\rangle+\frac{1}{2}\langle\nabla_{x}\bar{U}(s,\phi_{s}),q(\phi_{s})\nabla
_{x}\bar{U}(s,\phi_{s})\rangle\right) ds\right]
\label{Eq:ToBeBoundedAfterLimit2}
\end{align}

Next, we use the fact that $\bar{U}(s,x)$ is subsolution to the related homogenized HJB equation.  In particular, Definition \ref{Def:ClassicalSubsolution}, gives that
\begin{align}
\frac{d}{ds}\bar{U}(s,\phi_{s})&=\partial_{t}\bar{U}(s,\phi_{s})+\langle\nabla_{x}\bar{U}(s,\phi_{s}),\dot{\phi}_{s}\rangle\nonumber\\
&\geq \langle\dot{\phi}_{s}-r(\phi_{s}),\nabla_{x}\bar{U}(s,\phi_{s})\rangle
+\frac{1}{2}\langle\nabla_{x}\bar{U}(s,\phi_{s}),q(\phi_{s})\nabla_{x}\bar
{U}(s,\phi_{s})\rangle \label{Eq:LimitingTerm3}
\end{align}

Using inequality (\ref{Eq:LimitingTerm3}) and the third part of Definition \ref{Def:ClassicalSubsolution}, (\ref{Eq:ToBeBoundedAfterLimit2}) gives
\begin{align*}
&\liminf_{\epsilon\rightarrow0}-\epsilon\log Q^{\epsilon,\gamma}(t_{0},x_{0},y_{0};u) \geq \inf_{\phi,\phi_{t_{0}}=x_{0}}\left[ \int_{t_{0}}^{T}L(\phi_{s},\dot{\phi}_{s})ds +2h(\phi_{T})-h(\phi_{T})+\bar{U}(t_{0},x_{0})\right]\nonumber\\
&=G(t_{0},x_{0})+\bar{U}(t_{0},x_{0}),
\label{Eq:ToBeBoundedAfterLimit3}
\end{align*}
where the definition of $G(t_{0},x_{0})=\inf_{\phi,\phi_{t_{0}}=x_{0}}\left[ \int_{t_{0}}^{T}L(\phi_{s},\dot{\phi}_{s})ds +h(\phi_{T})\right]$ was used. This concludes the proof of the theorem.
\end{proof}

\section{An Illustrating Example}\label{S:Examples}

In this section we present a numerical example to demonstrate the main result of the paper. We consider the case of Example \ref{Ex:RandomCase} when both $X,Y$ are one-dimensional and with the additional choice $b=f$. To be more precise, we consider the system of slow-fast motion
\begin{eqnarray}
dX^{\epsilon}_{t}&=&\left[ - \frac{\epsilon}{\delta} \partial_{y}Q\left(Y^{\epsilon}_{t},\gamma\right)+c\left(  X^{\epsilon}_{t}%
,Y^{\epsilon}_{t},\gamma\right)\right]   dt+\sqrt{\epsilon \lambda}%
dW_{t},\nonumber\\
dY^{\epsilon}_{t}&=&\left[ - \frac{\epsilon}{\delta^{2}}\partial_{y}Q\left(Y^{\epsilon}_{t},\gamma\right)  +\frac{1}{\delta}g\left(  X^{\epsilon}_{t}%
,Y^{\epsilon}_{t},\gamma\right)\right] dt+\frac{\sqrt{\epsilon}}{\delta}\sqrt{2D}\left[
\theta
dW_{t}+\sqrt{1-\theta^{2}}dB_{t}\right], \label{Eq:MainExample2}\\
X^{\epsilon}_{0}&=&x_{0},\hspace{0.2cm}Y^{\epsilon}_{0}=y_{0}\nonumber
\end{eqnarray}
where $\theta\in[-1,1]$ is the correlation between the noises of the $X$ and $Y$ component and $D,\lambda$ are strictly positive constants. By Proposition \ref{P:NewMeasureRandomCase} we have that the corresponding unique ergodic invariant measure takes the form
\[
\pi(d\gamma)=\frac{e^{-\tilde{Q}(\gamma)/D}}{\mathbb{E}^{\nu}\left[e^{-\tilde{Q}(\gamma)/D}\right]}\nu(d\gamma).
\]

Let $Z=\mathbb{E}^{\nu}\left[e^{-\tilde{Q}(\gamma)/D}\right]$,  $K=\mathbb{E}^{\nu}\left[e^{\tilde{Q}(\gamma)/D}\right]$ and assume that $Z,K<\infty$.  It is easy to see that in the one-dimensional case we can solve (\ref{Eq:RandomCellProblem}) explicitly even with $\rho=0$. We obtain that the random field
\[
\chi(y,\gamma)=\frac{1}{K}\int_{0}^{y}e^{Q(z,\gamma)/D}dz-y
\]
solves (\ref{Eq:RandomCellProblem}) with $\rho=0$. One can also compute that $\mathbb{E}^{\nu}\left[\chi(y,\gamma)\right]=\mathbb{E}^{\nu}\left[\partial_{y}\chi(y,\gamma)\right]=0$. Moreover, $\chi(y,\gamma)$ is unique up to an additive constant and we easily compute that
\[
\partial_{y}\chi(y,\gamma)=\frac{1}{K}e^{Q(y,\gamma)/D}dz-1.
\]
Hence, if we write
$\psi(y,\gamma)=\partial_{y}\chi(y,\gamma)$, then $\psi(y,\gamma)=\tilde{\psi}(\tau_{y}\gamma)$ where $\tilde{\psi}(\gamma)=\frac{1}{K}e^{\tilde{Q}(\gamma)/D}-1$. Namely, we are in agreement with the setup in Section \ref{S:Notation}.

The large deviations rate function from Theorem \ref{T:MainTheorem3} takes the form
\begin{equation*}
S_{0T}(\phi)=%
\begin{cases}
\frac{1}{2}\int_{0}^{T}|\dot{\phi}_{s}-r(\phi_{s})|^{2}q^{-1}ds & \text{if }\phi\in\mathcal{AC}%
([0,T];\mathbb{R}) \text{  and } \phi_{0}=x_{0}\\
+\infty & \text{otherwise.}%
\end{cases}
\label{Eq:ActionFunctional1a}%
\end{equation*}
where the effective drift is
\begin{align*}
r(x)=\mathbb{E}^{\pi}\left[\tilde{c}(x,\cdot)+D\tilde{\chi}(\cdot)\tilde{g}(x,\cdot)\right]=\mathbb{E}^{\pi}\left[\tilde{c}(x,\cdot)-\tilde{g}(x,\cdot)\right] +\frac{1}{K}\mathbb{E}^{\pi}\left[\tilde{g}(x,\cdot)e^{\tilde{Q}(\cdot)/D}\right]
\end{align*}
and the effective diffusivity is computed to be
\begin{align*}
q&=\lambda+\left(2D-2\theta\sqrt{\lambda}\sqrt{2D}\right)\left(1- \frac{1}{KZ}\right).
\end{align*}

Notice that in the formulas above, what appears is $\partial_{y}\chi(y,\gamma)$ and not $\chi(y,\gamma)$.   Moreover,  H\"{o}lder's inequality guarantees that $KZ>1$, which then implies that the effective diffusivity $q>0$. Then, given an appropriate subsolution $\bar{U}(s,x)$ the asymptotically optimal change of measure is based on the control of Theorem \ref{T:UniformlyLogEfficient}, which takes the form
\begin{equation}
u(s,x,y,\gamma)=\left(u_{1}(s,x,y,\gamma),u_{2}(s,x,y,\gamma)\right)\label{Eq:ExampleControl1}
\end{equation}
with
\begin{align}
u_{1}(s,x,y,\gamma)&=-\left(\sqrt{\lambda}+\theta\sqrt{2D}\left(\frac{1}{K}e^{Q(y,\gamma)/D}-1\right)\right)\partial_{x}\bar{U}(s,x)\nonumber\\
u_{2}(s,x,y,\gamma)&=-\sqrt{2D}\sqrt{1-\theta^{2}}\left(\frac{1}{K}e^{Q(y,\gamma)/D}-1\right)\partial_{x}\bar{U}(s,x)\label{Eq:ExampleControl2}
\end{align}

Notice that the control is random itself, since it depends on the random field $Q(y,\gamma)$. For the simulations below we consider the case where $Q(y,\gamma)$ is a Gaussian random field. Even though this case is not immediately covered by our results, since Condition \ref{A:Assumption1} does not hold (lack of boundedness), the simulation results imply that the improvement in performance is still significant.

For the simulations below we  consider Gaussian random fields with mean zero and differentiable covariance function
\[
C(y)=e^{-|y|^{2}}.
\]

Moreover, for the numerical simulations, we consider the case $c(x,y,\gamma)=-V^{\prime}(x)$ where $V(x)$ is the quadratic potential $V(x)=\frac{1}{2}x^{2}$ and $g(x,y,\gamma)=0$.   Then, the effective drift simplifies to $r(x)=-x$. Let us assume that we want to compute
\begin{equation}
\theta(\epsilon,\gamma)=\mathbb{E}\left[e^{-\frac{1}{\epsilon}h(X^{\epsilon,\gamma}_{T})}\right]\label{Eq:ToBeEstimated}
\end{equation}
where
\[
h(x)=%
\begin{cases}
(x-1)^{2} & x\geq0\\
(x+1)^{2} & x<0.
\end{cases}
\]

The homogenized HJB equation (\ref{Eq:HJBequation2}) becomes
\begin{equation*}
G_{t}(t,x)- xG_{x}(t,x)-\frac{1}{2}q|G_{x}(t,x)|^{2}=0,\quad G(T,x)=h(x).
\label{Eq:1DHJBeqn}%
\end{equation*}
One can solve this variational problem explicitly and obtain
\[
G(t,x)=\frac{(e^{T}-|x|e^{ t})^{2}}{(1+q)e^{2T}-q e^{2 t}}.
\]

Notice that $G$ is not smooth at $x=0$, and therefore it is not a classical sense solution. To obtain a smooth subsolution, one should mollify it. However, as it follows from
 \cite{DupuisWang2, DupuisSpiliopoulosWang, VandenEijndenWeare}, the
bound on the performance is still valid given that the subsolution is the minimum of two
classical sense solutions with a single discontinuous interface. As a consequence, in this case we can just set  $\bar{U}=G$. Its gradient entering in the construction of the asymptotically optimal change of measure takes the form
\[
G_{x}(t,x)=-\frac{2e^{t}(e^{T}-|x|e^{ t})}{(1+q)e^{2T}-q e^{2 t}}\text{sign}(x).
\]
with $\text{sign}(x)=1$ if $x\geq 0$ and $\text{sign}(x)=-1$ if $x< 0$. In the numerical results of Table \ref{Table1} we report estimators of (\ref{Eq:ToBeEstimated}) based on the standard Monte-Carlo estimator (i.e., no change of measure), $\hat{\theta}_{0}(\epsilon)$ and based on the asymptotically optimal change of measure using the control (\ref{Eq:ExampleControl1})-(\ref{Eq:ExampleControl2}), $\hat{\theta}_{1}(\epsilon)$. Here,
\begin{equation}
\hat{\theta}_{1}(\epsilon)\doteq\frac{1}{N}\sum_{j=1}^{N}\left[e^{-\frac{1}{\epsilon}h(\bar{X}^{\epsilon,\gamma,j})}\frac{d\mathbb{P}}{d\bar{\mathbb{P}}}_{j}\right]\label{Def:OptimalEstimator1}
\end{equation}
where
\begin{equation}
\frac{d\mathbb{P}}{d\bar{\mathbb{P}}}_{j}\doteq e^{-\frac{1}{2\epsilon}\int_{0}^{T}\left\Vert u\left(s,\bar{X}^{\epsilon,\gamma,j}_{s},\bar{Y}^{\epsilon,\gamma,j}_{s}\right)\right\Vert^{2}ds
-\frac{1}{\sqrt{\epsilon}}\int_{0}^{T}\left<u\left(s,\bar{X}^{\epsilon,\gamma,j}_{s},\bar{Y}^{\epsilon,\gamma,j}_{s}\right),dZ_{j}(s)\right>}\label{Def:OptimalChangeOfMeasure1}
\end{equation}
and $(Z_{j},\bar{X}^{\epsilon,\gamma,j}_{s},\bar{Y}^{\epsilon,\gamma,j}_{s})$ is an independent sample generated from (\ref{Eq:Main2}) with  control $u$ given by (\ref{Eq:ExampleControl1})-(\ref{Eq:ExampleControl2}).

A few details on the simulation method that was followed are in order here. We concentrate on the  estimation of $\hat{\theta}_{1}(\epsilon)$, since the estimation of $\hat{\theta}_{0}(\epsilon)$ follows exactly the same steps but with control $u=0$. The algorithm first generates realizations of the  Gaussian random field $Q(y,\gamma)$ that has zero mean and covariance structure
$C(y)=e^{-|y|^{2}}$ and of its derivative $\partial_{y}Q(y,\gamma)$, on a fine grid of points. The randomization method, as outlined in \cite{Kramer2005} and \cite{Sabelfeld1991}, is used in order to simulate from $Q(y,\gamma)$ and $\partial_{y}Q(y,\gamma)$. 

Then the control $u$ given by (\ref{Eq:ExampleControl1})-(\ref{Eq:ExampleControl2}) with $\bar{U}_{x}(t,x)=G_{x}(t,x)$ is constructed, which is then inserted in (\ref{Eq:Main2}) with $b=f=-\partial_{y}Q$, $c(x,y,\gamma)=-x$, $g(x,y,\gamma)=0$, $\sigma(x,y,\gamma)=\sqrt{\lambda}$, $\tau_{1}(y,\gamma)=\sqrt{2D}\theta$ and $\tau_{2}(y,\gamma)=\sqrt{2D}\sqrt{1-\theta^{2}}$ in order to produce an independent sample of $(Z_{j},\bar{X}^{\epsilon,\gamma,j}_{s},\bar{Y}^{\epsilon,\gamma,j}_{s})$. We used the standard Euler-Maruyama method  to simulate the sample $(\bar{X}^{\epsilon,\gamma,j}_{s},\bar{Y}^{\epsilon,\gamma,j}_{s})$. The procedure is repeated independently $N$ times. Finally, the estimator $\hat{\theta}_{1}(\epsilon)$ given by (\ref{Def:OptimalEstimator1})-(\ref{Def:OptimalChangeOfMeasure1}) is returned.

The measure of performance is taken to be relative error per sample, $\rho_{0}(\epsilon)$ and $\rho_{1}(\epsilon)$ for $\hat{\theta}_{0}(\epsilon)$ and $\hat{\theta}_{1}(\epsilon)$ respectively, which is defined as follows:
\[
\mbox{relative error per sample} \doteq\sqrt{N}\frac{\mbox{standard deviation of the estimator}}%
{\mbox{expected value of the estimator}}.
\]

The smaller the relative error per sample is, the more efficient the estimator is. However, in practice both the standard deviation and the expected value
of an estimator are typically unknown, which implies that empirical relative error is
often used for measurement. This means that, the expected value of
the estimator will be replaced by the empirical sample mean, and the
standard deviation of the estimator will be replaced by the
empirical sample standard error.

In Table \ref{Table1} we see simulation results and estimates for $\hat{\theta}_{0}(\epsilon)$ and $\hat{\theta}_{1}(\epsilon)$ as well as for their corresponding estimated relative errors per sample $\hat{\rho}_{0}(\epsilon)$ and $\hat{\rho}_{1}(\epsilon)$. Different pairs of $(\epsilon,\delta)$ are used and the parameters used in the simulation are $(T,\lambda,D,\theta)=(1,1,1,1/2)$ and initial point $(x_{0},y_{0})=(0.05,0)$.

\begin{table}[th]
\begin{center}%
\begin{small}
\begin{tabular}
[c]{|c|c|c|c|c|c|c|c|}\hline
No. & $\epsilon$ & $\delta$ & $\epsilon/\delta$ & $\hat{\theta}_{0}(\epsilon)$& $\hat{\theta}_{1}(\epsilon)$ &  $\hat{\rho}_{0}(\epsilon)$ & $\hat{\rho}_{1}(\epsilon)$ \\\hline
$1$ & $0.125$ & $0.04$ & $3.125$ & $3.20e-2$ & $3.22e-2$  &$2.65$ & $3.43$\\\hline
$2$ & $0.0625$ & $0.018$ & $3.472$ & $8.45e-4$ & $8.34e-4$ & $15.23$ & $3.39$ \\\hline
$3$ & $0.05$ & $0.013$ & $3.846$ & $1.27e-4$ & $1.35e-4$ & $17.39$ & $2.51$ \\\hline
$4$ & $0.03125$ & $0.008$ & $3.91$ & $2.96e-7$ & $5.94e-7$ &$53.78$ & $2.32$ \\\hline
$5$ & $0.025$ & $0.005$ & $5.0$ & $1.28e-9$ & $1.53e-8$  &$95.43$ & $2.24$ \\\hline
$6$ & $0.02$ & $0.003$ & $6.67$ & $1.73e-11$ & $2.85e-10$ &$171.23$ & $2.54$ \\\hline
\end{tabular}
\end{small}
\end{center}
\caption{Comparison table:  standard Monte-Carlo versus optimal change of measure. Parameters used are $(T,\lambda,D,\theta)=(1,1,1,1/2)$ and initial point $(x_{0},y_{0})=(0.05,0)$.}%
\label{Table1}%
\end{table}

Table \ref{Table1} shows that the change of measure using the control $u$ from (\ref{Eq:ExampleControl1})-(\ref{Eq:ExampleControl2}) significantly outperforms the estimator with no change of measure, i.e., standard Monte Carlo. In addition, we note that the estimator $\hat{\theta}_{1}(\epsilon)$ seems to be of bounded relative error, which is a stronger notion of efficiency than asymptotic optimality.

Simulations were done using parallel computing in the C programming language (MPI).
We used Mersenne-Twister \cite{MatsumotoNishimura1998} for the random number
generator, with a sample size of $N=6* 10^{6}$. The discretization error is taken to be $\zeta=0.001$. The presence of the fast scales and the fact that our interest is in rare events, which also implies small $\epsilon$, have significant implications on the choice of the discretization step, $\Delta t$.
In particular, as it is derived in a related periodic setting in \cite{DupuisSpiliopoulosWang}, we have the relation
\[
\Delta t=O\left(  \frac{\delta^{2}}{\epsilon}\zeta^{1/p}\right)
\]
where $p$ is the weak order of convergence of the direct discretization scheme being used. In the preceding simulation we used standard Euler-Maruyama scheme, i.e., $p=1$.
As $\epsilon,\delta$ decrease, $\Delta t$ should also decrease in order to maintain the same approximation error, quantified by $\zeta$. In turn, this increases the computational cost significantly. The increasing computational cost highlights the importance of
 fast simulation techniques (such as importance sampling that is studied in this paper) for treating rare event problems in multiscale environments. Namely, Table \ref{Table1} essentially shows that for the same level of accuracy one needs to simulate significantly less number of trajectories when using the optimal change of simulation measure, which in turn could imply significant gain in computational costs.

\section{Conclusions}
In this paper we have considered systems of SDE's with fast and slow component and small noise. We have assumed that the coefficients of the diffusion processes are themselves stationary and ergodic random fields living in some probability space. Our interest is in constructing efficient Monte-Carlo methods for the estimation of rare event probabilities or of expectations of functionals related to rare events that are provably efficient with probability $1$ with respect to the random environment. Using the recently developed related quenched large deviations results of \cite{Spiliopoulos2014}, we construct such schemes and prove their quenched asymptotic optimality. The construction is based on using the gradient of the solution to the corrector from random homogenization theory as well as to subsolutions to the related homogenized HJB equation.

The paper provides a theoretical framework that allows to rigorously assess the performance of importance sampling Monte-Carlo schemes for multiscale diffusions in random environments. Future work on this topic include (a): investigation of the effect of numerical approximations to the gradient of the solution to the macroscopic problem (\ref{Eq:RandomCellProblem}) on the simulation algorithm for the case $b\neq 0$ (recall that by Remark \ref{R:NoUnboundedDriftTerm} when $b=0$, $\chi=0$ and thus the macroscopic problem (\ref{Eq:RandomCellProblem}) is not used), and (b):  explicit construction of subsolutions to the associated homogenized HJB equations. The latter has been partially tackled when one has metastability issues and to be more specific for the problem of escape from an equilibrium, see \cite{DupuisSpiliopoulosZhou2013,DupuisSpiliopoulos2014}.

\section{Acknowledgements}
The author was partially supported by the National Science Foundation (DMS 1312124).

\appendix
\section{Auxiliary technical results}\label{App:AuxiliaryTechnicalResults}

\begin{lemma}\label{L:RestrictionControls}
Assume Conditions \ref{A:Assumption1}-\ref{A:Assumption2}. The infimum of the representation in (\ref{Eq:ToBeBounded}) can be taken over all controls that satisfy Condition
\ref{Eq:RestrictionControls}.
\end{lemma}

\begin{proof}[Proof of Lemma \ref{L:RestrictionControls}]

Without loss of generality, we can consider a function $h(x)$ that is bounded and uniformly Lipschitz
continuous in $\mathbb{R}^{m}$. Namely, there exists a constant $L_{h}$ such that
\[
|h(x)-h(y)|\leq L_{h}\left\Vert x-y\right\Vert
\]
and $\left\Vert h\right\Vert_{\infty}=\sup_{x\in\mathbb{R}^{m}}|h(x)|<\infty$. Generality due to assuming Lipschitz continuity instead of continuity is not lost due to a standard approximation argument, see for example \cite{DupuisEllis}. Let us also  assume for notational convenience $t_{0}=0$.

Recall that $(\bar{X}^{\epsilon},\bar{Y}^{\epsilon})=(\bar{X}^{\epsilon,\gamma},\bar{Y}^{\epsilon,\gamma})$ depend on the random medium $\gamma\in \Gamma$ and that the pair satisfies (\ref{Eq:Main2}). However, for notational convenience, we shall omit this in the notation. We start by observing that the definition of $Q^{\epsilon,\gamma}(0,x_{0},y_{0};u)$ by (\ref{Eq:2ndMoment1}) and Lemma 4.3 of \cite{DupuisSpiliopoulosWang} give
\begin{align*}
Q^{\epsilon}(0,x_{0},y_{0};u)&=\bar{\mathbb{E}}\left[  \exp\left\{
-\frac{2}{\epsilon}h(X^{\epsilon}_{T})\right\}  \left(  \frac{d\mathbb{P}%
}{d\bar{\mathbb{P}}}(X^{\epsilon}, Y^{\epsilon})\right)  ^{2}\right]\nonumber\\
&=
\mathbb{E}\left[  \exp\left\{
-\frac{2}{\epsilon}h(\bar{X}^{\epsilon}_{T})+\frac{1}{\epsilon} \int_{0}^{T}\left\Vert u_{\rho}(s,\bar{X}^{\epsilon}_{s},\bar{Y}^{\epsilon}_{s})\right\Vert^{2} \right\} \right]
\label{Eq:2ndMoment1_a0}%
\end{align*}

Therefore, we get
\begin{equation*}
-\epsilon\log \bar{\mathbb{E}}\left[  \exp\left\{
-\frac{2}{\epsilon}h(X^{\epsilon}_{T})\right\}  \left(  \frac{d\mathbb{P}%
}{d\bar{\mathbb{P}}}(X^{\epsilon}, Y^{\epsilon})\right)  ^{2}\right]\leq
2 \mathbb{E}\left[h(\bar{X}^{\epsilon}_{T})\right]-\mathbb{E} \int_{0}^{T}\left\Vert u_{\rho}(s,\bar{X}^{\epsilon}_{s},\bar{Y}^{\epsilon}_{s})\right\Vert^{2}ds
\label{Eq:2ndMoment1_b}%
\end{equation*}

Using the representation (\ref{Eq:ToBeBounded}), we only need to consider controls $v(s)$ such that
\begin{align*}
\mathbb{E}\left[  \frac{1}{2}\int_{0}^{T}\left\Vert
v(s)\right\Vert ^{2}ds-\int_{0}^{T}\Vert u_{\rho}(s,\hat{X}^{\epsilon}_{s},\hat{Y}^{\epsilon}
_{s})\Vert^{2}ds+2h(\hat{X}^{\epsilon}_{T})\right] \leq 2 \mathbb{E}\left[h(\bar{X}^{\epsilon}_{T})\right]-\mathbb{E} \int_{0}^{T}\left\Vert u_{\rho}(s,\bar{X}^{\epsilon}_{s},\bar{Y}^{\epsilon}_{s})\right\Vert^{2}ds
\end{align*}
where $\left(\hat{X}^{\epsilon}_{s},\hat{Y}^{\epsilon}_{s}\right)$ satisfies (\ref{Eq:Main3}) whereas $\left(\bar{X}^{\epsilon}_{s},\bar{Y}^{\epsilon}_{s}\right)$ satisfies (\ref{Eq:Main2}). After rearranging the latter display and using  the Lipschitz continuity of $h$ and of $\left\Vert u_{\rho}\right\Vert^{2}$ (which follows by Condition \ref{A:Assumption1} and \cite{Rhodes2009a}), we get
\begin{align*}
\mathbb{E}\left[  \frac{1}{2}\int_{0}^{T}\left\Vert
v(s)\right\Vert ^{2}ds\right]&\leq 2\mathbb{E}\left|h(\hat{X}^{\epsilon}_{T})-h(\bar{X}^{\epsilon}_{T})\right| +\mathbb{E}\int_{0}^{T}\left|\left\Vert u_{\rho}(s,\hat{X}^{\epsilon}_{s},\hat{Y}^{\epsilon}_{s})\right\Vert^{2}-\left\Vert u_{\rho}(s,\bar{X}^{\epsilon}_{s},\bar{Y}^{\epsilon}_{s})\right\Vert^{2}\right|ds\nonumber\\
&\leq C_{0}\left\{\mathbb{E}\left\Vert\hat{X}^{\epsilon}_{T}-\bar{X}^{\epsilon}_{T}\right\Vert+\mathbb{E}\int_{0}^{T}\left[\left\Vert\hat{X}^{\epsilon}_{s}-\bar{X}^{\epsilon}_{s}\right\Vert
+\left\Vert\hat{Y}^{\epsilon}_{s}-\bar{Y}^{\epsilon}_{s}\right\Vert\right]ds\right\}
\end{align*}
where $C_{0}<\infty$ is a constant that does not depend on $\epsilon$ or $\delta$. By changing time, we then have
\begin{align}
\frac{1}{\epsilon}\mathbb{E}\left[  \frac{1}{2}\int_{0}^{\frac{\delta^{2}}{\epsilon}T}\left\Vert
v(s)\right\Vert ^{2}ds\right]
&\leq C_{0}\left\{\frac{1}{\epsilon}\mathbb{E}\left\Vert\hat{X}^{\epsilon}_{\frac{\delta^{2}}{\epsilon}T}-\bar{X}^{\epsilon}_{\frac{\delta^{2}}{\epsilon}T}\right\Vert+\frac{1}{\epsilon}\mathbb{E}\int_{0}^{\frac{\delta^{2}}{\epsilon}T}
\left[\left\Vert\hat{X}^{\epsilon}_{s}-\bar{X}^{\epsilon}_{s}\right\Vert+\left\Vert\hat{Y}^{\epsilon}_{s}-\bar{Y}^{\epsilon}_{s}\right\Vert\right]ds\right\}\nonumber\\
&=C_{0}\left\{\frac{\delta}{\epsilon}\mathbb{E}\left\Vert\frac{1}{\delta}\hat{X}^{\epsilon}_{\frac{\delta^{2}}{\epsilon}T}-\frac{1}{\delta}\bar{X}^{\epsilon}_{\frac{\delta^{2}}{\epsilon}T}\right\Vert\right.\nonumber\\
&\left. \qquad+\left(\frac{\delta}{\epsilon}\right)^{2}
\mathbb{E}\int_{0}^{T}\left[\left\Vert\hat{X}^{\epsilon}_{\frac{\delta^{2}}{\epsilon} s}-\bar{X}^{\epsilon}_{\frac{\delta^{2}}{\epsilon} s}\right\Vert+\left\Vert\hat{Y}^{\epsilon}_{\frac{\delta^{2}}{\epsilon} s}-\bar{Y}^{\epsilon}_{\frac{\delta^{2}}{\epsilon} s}\right\Vert\right]ds\right\}
\label{Eq:2ndMoment1_e}%
\end{align}

Let us then define
\[
\nu^{\epsilon}_{T}\doteq \frac{1}{\epsilon}\mathbb{E}\left[  \int_{0}^{\frac{\delta^{2}T}{\epsilon}}\left\Vert v(s)\right\Vert^{2}ds\right]
\]
and
\[
m^{\epsilon}_{s}\doteq \mathbb{E}\left\Vert\frac{1}{\delta}\hat{X}^{\epsilon}_{\frac{\delta^{2}}{\epsilon} s}-\frac{1}{\delta}\bar{X}^{\epsilon}_{\frac{\delta^{2}}{\epsilon} s}\right\Vert+\mathbb{E}\left\Vert\hat{Y}^{\epsilon}_{\frac{\delta^{2}}{\epsilon} s}-\bar{Y}^{\epsilon}_{\frac{\delta^{2}}{\epsilon} s}\right\Vert
\]

After these definitions, (\ref{Eq:2ndMoment1_e}) gives us via Jensen's inequality
\begin{align}
\left|\nu^{\epsilon}_{T}\right|^{2}\leq C_{1}\left\{\left(\frac{\delta}{\epsilon}\right)^{2}m^{\epsilon}_{T}+\left(\frac{\delta}{\epsilon}\right)^{4}\int_{0}^{T}\left|m^{\epsilon}_{s}\right|^{2}ds\right\}\label{Eq:BoundForNu}
\end{align}
for an unimportant constant $C_{1}<\infty$. But then it is easy to see that
\begin{align*}
m^{\epsilon}_{T}&\leq C_{2}\left\{\int_{0}^{T} m^{\epsilon}_{s}ds+
 \left\{\left|\frac{\delta}{\epsilon}\mathbb{E} \left[\int_{0}^{T}\left\Vert v_{1}\left(\frac{\delta^{2}s}{\epsilon}\right)\right\Vert ds\right]\right|^{2}
 +\left|\frac{\delta}{\epsilon}\mathbb{E} \left[\int_{0}^{T}\left\Vert v_{2}\left(\frac{\delta^{2}s}{\epsilon}\right)\right\Vert ds\right]\right|^{2}\right\}\right\}
\end{align*}
where the constant $C_{2}<\infty$ depends only on the Lipschitz constants of $b,c,f,g,\sigma,\tau_{1},\tau_{2}$ and on the sup norm of $\sigma,\tau_{1},\tau_{2}$.  Gronwall lemma, gives us
\begin{align*}
m^{\epsilon}_{T}&\leq C_{2}\left|a^{\epsilon}_{T}\right|^{2}+ C^{2}_{2}\int_{0}^{T}\left|a^{\epsilon}_{s}\right|^{2}e^{C_{2}(T-s)}ds
\end{align*}
where we have defined
\begin{align*}
\left|a^{\epsilon}_{T}\right|^{2}&\doteq\left|\frac{\epsilon}{\delta}\frac{1}{\epsilon}\mathbb{E}\left[  \int_{0}^{\frac{\delta^{2}T}{\epsilon}}\left\Vert v_{1}\left(s\right)\right\Vert ds\right]\right|^{2}+\left|\frac{\epsilon}{\delta}\frac{1}{\epsilon}\mathbb{E}\left[  \int_{0}^{\frac{\delta^{2}T}{\epsilon}}\left\Vert v_{2}\left(s\right)\right\Vert ds\right]\right|^{2}\nonumber\\
\end{align*}
Since, H\"{o}lder inequality followed by Young's inequality give
\begin{align*}
\left|a^{\epsilon}_{T}\right|^{2}
&\leq \frac{1}{\delta^{2}}\frac{\delta^{2}T}{\epsilon}\mathbb{E}\left[  \int_{0}^{\frac{\delta^{2}T}{\epsilon}}\left\Vert v\left(s\right)\right\Vert^{2}ds\right]= T\nu^{\epsilon}_{T}\leq \frac{T^{2}}{2}+\frac{\left|\nu^{\epsilon}_{T}\right|^{2}}{2},
\end{align*}
we subsequently obtain
\begin{align*}
m^{\epsilon}_{T}&\leq C_{2}\left(\frac{T^{2}}{2}+\frac{\left|\nu^{\epsilon}_{T}\right|^{2}}{2}\right)+ C^{2}_{2}\int_{0}^{T}\left(\frac{s^{2}}{2}+\frac{\left|\nu^{\epsilon}_{s}\right|^{2}}{2}\right)e^{C_{2}(T-s)}ds.
\end{align*}

From now on, we omit distinguishing between unimportant constants and $C_{0},C_{1}, C_{2}$ will all be denoted by $C$. However, the constant $C$ may change from line to line, but this is not explicitly mentioned in the notation.
Putting the latter estimate back in (\ref{Eq:BoundForNu}) we get for $\delta/\epsilon$ sufficiently small
\begin{align*}
\left|\nu^{\epsilon}_{T}\right|^{2}& \leq  C \left|\frac{\delta}{\epsilon}\right|^{2} \left[ \left(\frac{T^{2}}{2}+\frac{\left|\nu^{\epsilon}_{T}\right|^{2}}{2}\right)+ C\int_{0}^{T}\left(\frac{s^{2}}{2}+\frac{\left|\nu^{\epsilon}_{s}\right|^{2}}{2}\right)e^{C(T-s)}ds\right]+\nonumber\\
&\quad C\left(\frac{\delta}{\epsilon}\right)^{4}\int_{0}^{T}
\left[\left(\frac{s^{2}}{2}+\frac{\left|\nu^{\epsilon}_{s}\right|^{2}}{2}\right)+ \int_{0}^{s}\left(\frac{\theta^{2}}{2}+\frac{\left|\nu^{\epsilon}_{\theta}\right|^{2}}{2}\right)e^{C(s-\theta)}d\theta\right]ds\nonumber\\
&\leq C\left(\frac{\delta}{\epsilon}\right)^{2}\left[\frac{T^{2}}{2}e^{CT}+C\int_{0}^{T}\frac{|\nu^{\epsilon}_{s}|}{2}e^{C(T-s)}ds\right]
\end{align*}

Thus, we have
\begin{align*}
e^{-CT}\frac{\left|\nu^{\epsilon}_{T}\right|^{2}}{2}
&\leq C \left|\frac{\delta}{\epsilon}\right|^{2} \left[   \frac{T^{2}}{2} +C \int_{0}^{T}e^{-Cs}\frac{\left|\nu^{\epsilon}_{s}\right|^{2}}{2}ds\right]
\end{align*}

Thus, for $\delta,\epsilon$ small enough such that $\delta/\epsilon \ll 1$, Gronwall inequality gives us

\[
\left|\nu^{\epsilon}_{T}\right|^{2}\leq C (\delta/\epsilon)^{2},
\]
where the constant $C$ does not depend on $\epsilon,\delta$ or $\gamma$. This concludes the proof of the lemma.
\end{proof}

\end{document}